\documentclass[12pt]{amsart}
\usepackage{amssymb}
\usepackage[all]{xy}
\usepackage{lscape}

\theoremstyle{plain} 
\newtheorem{theor}[equation]{Theorem}
\newtheorem{cor}[equation]{Corollary}
\newtheorem{lem}[equation]{Lemma}
\newtheorem{conjecture}[equation]{Conjecture}
\newtheorem{statement}[equation]{Statement}
\newtheorem{proposition}[equation]{Proposition}

\theoremstyle{definition}
\newtheorem{defin}[equation]{Definition}

\theoremstyle{remark}
\newtheorem{rem}[equation]{Remark}
\newtheorem{ex}[equation]{Example}

\newtheorem{propriete}[equation]{Property}

\newcommand{\BV}{\Delta}
\newcommand{\sectiontriviale}{s}

\def\build#1_#2^#3{\mathrel{\mathop{\kern0pt#1}\limits_{#2}^{#3}}}

\begin{document}
\begin{abstract}
Let $M$ be a path-connected closed oriented $d$-dimensional smooth manifold
and let ${\Bbbk}$ be a principal ideal domain.
By Chas and Sullivan, the shifted free loop space homology of $M$, $H_{*+d}(LM)$
is a Batalin-Vilkovisky algebra.
Let $G$ be a topological group such that $M$ is a classifying space of $G$.
Denote by $S_*(G)$ the (normalized) singular chains on $G$.
Suppose that $G$ is discrete or path-connected.
We show that there is a Van Den Bergh type isomorphism
$$
HH^{-p}(S_*(G),S_*(G))\cong HH_{p+d}(S_*(G),S_*(G)).
$$
Therefore, the Gerstenhaber algebra $HH^{*}(S_*(G),S_*(G))$
is a Batalin-Vilkovisky algebra and we have a linear isomorphism
$$HH^{*}(S_*(G),S_*(G))\cong H_{*+d}(LM).$$
This linear isomorphism is expected to be an isomorphism of Batalin-Vilkovisky algebras.
We also give a new  characterization of Batalin-Vilkovisky algebra in term of derived bracket.
\end{abstract}
\title{\bf Van Den Bergh isomorphisms in String Topology}
\author{Luc Menichi}
\address{UMR 6093 associ\'ee au CNRS\\
Universit\'e d'Angers, Facult\'e des Sciences\\
2 Boulevard Lavoisier\\49045 Angers, FRANCE}
\email{firstname.lastname at univ-angers.fr}

\maketitle
\begin{center}
\end{center}

\section{Introduction}
We work over an arbitrary
principal ideal domain ${\Bbbk}$.
Let $M$ be a compact oriented $d$-dimensional smooth manifold.
Denote by $LM:=map(S^1,M)$ the free loop space on $M$.
Chas and Sullivan~\cite{Chas-Sullivan:stringtop} have shown that the shifted free loop homology
$H_{*+d}(LM)$ has a structure of Batalin-Vilkovisky algebra
(Definition~\ref{definition BV algebre}). In particular, they showed
that $H_{*+d}(LM)$ is a Gerstenhaber algebra (Definition~\ref{definition algebre de Gerstenhaber}).
On the other hand, let $A$ be a differential graded (unital
associative) algebra.
The Hochschild cohomology of $A$ with coefficients in $A$, $HH^*(A,A)$,
is a Gerstenhaber algebra.
These two Gerstenhaber algebras are expected to be related:
\begin{conjecture}\label{iso gerstenhaber goodwillie}
Let $G$ be a topological group such that $M$ is a classifying space of $G$.
There is an isomorphism of Gerstenhaber algebras
$H_{*+d}(LM)\cong HH^*(S_*(G),S_*(G))$ between the free loop space homology and the Hochschild cohomology of the
algebra of singular chains on $G$.
\end{conjecture}

Suppose that $G$ is discrete or path-connected.
In this paper, we define a Batalin-Vilkovisky algebra structure
on $HH^*(S_*(G), S_*(G))$
and an isomorphism of graded ${\Bbbk}$-modules
$$
BFG^{-1}\circ \mathcal{D}:H_{*+d}(LM)\cong HH^*(S_*(G), S_*(G))
$$
which is compatible with the two $\Delta$ operators of the two Batalin-Vilkovisky algebras:
$BFG^{-1}\circ \mathcal{D}\circ\Delta=\Delta\circ BFG^{-1}\circ \mathcal{D}$.
Indeed, Burghelea, Fiedorowicz~\cite{Burghelea-Fiedorowicz:chak} and Goodwillie~\cite{Goodwillie:cychdfl} gave an isomorphism of graded ${\Bbbk}$-modules
$$
BFG:HH_*(S_*(G), S_*(G))\buildrel{\cong}\over\rightarrow H_{*}(LM)
$$
which interchanges Connes boundary map $B$ and the $\Delta$ operator
on $H_{*+d}(LM)$: $BFG\circ B=\Delta\circ BFG$. And in this paper, our main result is:
\begin{theor}\label{structure BV sur la cohomologie de Hochschild de G}
(Theorems~\ref{cas discret} and~\ref{cas connexe})
Let $G$ be a discrete or a path-connected topological group
such that its classifying space $BG$ is an oriented Poincar\'e duality space
of formal dimension $d$.
Then 

a) there exists ${\Bbbk}$-linear isomorphisms

$$\mathcal{D}:HH_{d-p}(S_*(G),S_*(G))\buildrel{\cong}\over\rightarrow HH^p(S_*(G),S_*(G)).$$

b) If $B$ denotes Connes boundary map
on $HH_*(S_*(G),S_*(G))$ then $\Delta:=-\mathcal{D}\circ B\circ \mathcal{D}^{-1}$
defines a structure of Batalin-Vilkovisky algebra on
$HH^*(S_*(G),S_*(G))$, extending the canonical Gerstenhaber algebra structure.

c) The cyclic homology of $S_*(G)$, $HC_*(S_*(G))$ has a Lie bracket
of degre $2-d$.
\end{theor}
By~\cite[Proposition 28]{Menichi:BV_Hochschild}, c) follows directly from b). 
Note that when $G$ is a discrete group, the algebra of normalized singular chains on $G$, $S_*(G)$ is just the group ring ${\Bbbk}[G]$.

To prove Conjecture~\ref{iso gerstenhaber goodwillie} in the discrete or
path-connected case,
it suffices now to show that
the composite $BFG^{-1}\circ \mathcal{D}$ is a morphism of graded algebras.
When ${\Bbbk}$ is a field of characteristic $0$ and $G$ is discrete,
this was proved by Vaintrob~\cite{Vaintrob:stBVaHcohGoldb}.

Suppose now that
\begin{equation}\label{corps et simplement connexe}
\text{$M$ is simply-connected and that ${\Bbbk}$ is a field.}
\end{equation}
In this case, there is a more famous dual conjecture relating
Hochschild cohomology and string topology.
\begin{conjecture}\label{conjecture iso algebre de Gerstenhaber}
\noindent Under~(\ref{corps et simplement connexe}),
there is an isomorphism of Gerstenhaber algebras
$H_{*+d}(LM)\cong HH^*(S^*(M),S^*(M))$ between the free loop space homology and the Hochschild cohomology of the
algebra of singular cochains on $M$.
\end{conjecture}
And in fact, Theorem~\ref{structure BV sur la cohomologie de Hochschild de G}
is the Eckmann-Hilton or Koszul dual of the following theorem.
\begin{theor}(\cite[Theorem 23]{Felix-Thomas-Vigue:Hochschildmanifold}
and~\cite[Theorem 22]{Menichi:BV_Hochschild})\label{structure BV sur la cohomologie de Hochschild des cochaines}
Assume~(\ref{corps et simplement connexe}).

a) There exist isomorphism of graded ${\Bbbk}$-vector spaces
$$
FTV:HH^{p-d}(S^*(M),S^*(M)^\vee)\buildrel{\cong}\over\rightarrow
HH^{p}(S^*(M),S^*(M)).
$$

b) The Connes coboundary $B^\vee$ on $HH^*(S^*(M),S^*(M)^\vee)$
defines via the isomorphism $FTV$
a structure of Batalin-Vilkovisky algebra extending
the Gerstenhaber algebra $HH^*(S^*(M),S^*(M))$.
\end{theor}
Jones~\cite{JonesJ:Cycheh} proved that there is an isomorphism
$$J:H_{p+d}(LM)\buildrel{\cong}\over\rightarrow
HH^{-p-d}(S^*(M),S^*(M)^\vee)$$
such that the $\Delta$ operator of the Batalin-Vilkovisky algebra
$H_{*+d}(LM)$ and Connes coboundary map $B^\vee$ on
$HH^{*-d}(S^*(M),S^*(M)^\vee)$ satisfies $J\circ\Delta=B^\vee\circ J$.
Therefore, as we explain in~\cite{Menichi:BV_Hochschild}, to prove conjecture~\ref{conjecture iso algebre de Gerstenhaber}, it suffices to show that
the composite $FTV\circ J$ is a morphism of graded algebras.

In~\cite{Felix-Menichi-Thomas:GerstduaiHochcoh},
together with Felix and Thomas, we prove that Hochschild
cohomology satisfies some Eckmann-Hilton or Koszul duality.
\begin{theor}~\cite[Corollary 2]{Felix-Menichi-Thomas:GerstduaiHochcoh}\label{iso duality in Gerstenhaber}(See also~\cite[Theorem 69 and below]{Chataur-Menichi:stringclass})
Let ${\Bbbk}$ be a field.
Let $G$ be a connected topological group.
Denote by $S^*(BG)$ the algebra of singular cochains on the classifying
space of $G$.
Suppose that for all $i\in\mathbb{N}$, $H_i(BG)$ is finite dimensional.
Then there exists an isomorphism of Gerstenhaber algebras
$$
Gerst:HH^*(S_*(G),S_*(G))\buildrel{\cong}\over\rightarrow HH^*(S^*(BG),S^*(BG)).
$$
\end{theor}
Therefore under~(\ref{corps et simplement connexe}),
Conjectures~\ref{conjecture iso algebre de Gerstenhaber}
and~\ref{iso gerstenhaber goodwillie} are equivalent
and under~(\ref{corps et simplement connexe}),
Theorem~\ref{structure BV sur la cohomologie de Hochschild de G} as stated in
this introduction follows from Theorem~\ref{structure BV sur la cohomologie de Hochschild des cochaines}.

The problem is that the isomorphism $Gerst$ in Theorem~\ref{iso duality in Gerstenhaber} does not admit a simple formula.
On the contrary, as we explain in Theorems~\ref{cas discret} and~\ref{cas connexe}, in this paper, the isomorphism $\mathcal{D}$ is very simple:
$\mathcal{D}^{-1}$ is given by the cap product with a fundamental class
$c\in HH_d(S_*(G),S_*(G))$.

In~\cite[Theorem 3.4.3 i)]{Ginzburg:CYalgebras}, Ginzburg
(See also~\cite[Proposition 1.4]{Lambre:dualiteVdB}) shows that
for any Calabi-Yau algebra $A$, the Van den Bergh duality isomorphism
$\mathcal{D}:HH_{d-p}(A,A)\buildrel{\cong}\over\rightarrow HH^p(A,A)$
is $HH^*(A,A)$-linear: $\mathcal{D}^{-1}$ is also given by the cap product with a fundamental class
$c\in HH_d(A,A)$.

We now give the plan of the paper:

{\bf Section 2:}
We recall the definitions of the Bar construction, of the Hochschild
(co)chain complex and of Hochschild (co)homology.

{\bf Section 3:}
We show that, for some augmented differential graded algebra $A$ such that
the dual of its reduced bar construction $B(A)^\vee$ satisfies Poincar\'e duality,
we have a Van den Bergh duality isomorphism
$HH_{d-p}(A,A)\cong HH^p(A,A)$ if $A$ is connected (Corollaries~\ref{cas simplement connexe} and~\ref{cas cochain connexe}).

{\bf Section 4:}
There is a well known isomorphism between group (co)homology and
Hochschild (co)homology. We show that, through this isomorphism,
cap products in Hochschild (co)homology
correspond to cap products in group (co)homology.

{\bf Section 5:}
We give a new characterization of Batalin-Vilkovisky algebras.

{\bf Section 6:}
Ginzburg proved that if Hochschild (co)homology satisfies
a Van den Bergh duality isomorphism $HH_{d-p}(A,A)\cong HH^p(A,A)$
then Hochschild cohomology has a Batalin-Vilkovisky algebra structure.
We rewrite the proof of Ginzburg using our new characterization of
Batalin-Vilkovisky algebras and insisting on signs.

{\bf Section 7:}
We show that a differential graded algebra quasi-isomorphic to
an algebra satisfying Poincar\'e duality, also satisfies Poincar\'e
duality
(Proposition~\ref{invariance homotopique algebre a dualite de Poincare}).
Finally, we show our main theorem for path-connected topological group.

{\bf Section 8:}
We show our main theorem for discrete groups.
Extending a result of Kontsevich~\cite[Corollary 6.1.4]{Ginzburg:CYalgebras}
and Lambre~\cite[Lemme 6.2]{Lambre:dualiteVdB},
we also show that, over any commutative ring ${\Bbbk}$, the group ring
${\Bbbk}[G]$ of an orientable Poincar\'e duality group is a Calabi-Yau algebra.

{\bf Section 9:}
Let $G$ be a path-connected compact Lie group of dimension $d$.
We give another Van Den Bergh type isomorphism
$$
HH^{p}(S^*(BG),S^*(BG))\cong HH_{-d-p}(S^*(BG),S^*(BG)).
$$
Therefore, the Gerstenhaber algebra $HH^{*}(S^*(BG),S^*(BG))$
is a Batalin-Vilkovisky algebra and we have a linear isomorphism
$$HH^{*}(S^*(BG),S^*(BG))\cong H^{*+d}(LBG).$$

{\bf Appendix:}
We recall the notion of derived bracket following Kosmann-Schwarzbach
~\cite{Kosmann-Schwarzbach:PoissontoGerst}.
We interpret our new characterization of Batalin-Vilkovisky algebra
in term of derived bracket (Theorem~\ref{injection d'une BV-algebre dans End}).
To any differential graded algebra $A$, we associate

-a new Lie bracket on $A$ (Remark~\ref{crochet de Lie associe a une adg}),

-a new Gerstenhaber algebra which is a sub algebra of the 
endomorphism algebra of $HH_*(A,A)$
(Theorem~\ref{hochschild morphism de gerstenhaber}).

We conjecture that
Theorem~\ref{structure BV sur la cohomologie de Hochschild de G}
is true without assuming that $G$ is discrete or path-connected.
Note that the proof of the discrete case (Sections 4 and 8) is independent
of the proof of the path-connected case (Sections 3 and 7).

{\em Acknowledgment:}
We wish to thank Jean-Claude Thomas for several discussions, in particular for pointing the Mittag-Leffler
condition which is the key of Proposition~\ref{quasi-iso du cap apres limite inverse}.

\section{Hochschild homology and cohomology}

We work over an arbitrary commutative ring ${\Bbbk}$ except in sections 3 and 7, where ${\Bbbk}$ is assumed to be a principal ideal domain and in section 9
where ${\Bbbk}$ is assumed to be a field.
We use the graded differential algebra of~\cite[Chapter
3]{Felix-Halperin-Thomas:ratht}.
In particular,
an element of lower degree $i\in\mathbb{Z}$ is by the
{\it classical convention}~\cite[p. 41-2]{Felix-Halperin-Thomas:ratht}
of upper degree $-i$. Differentials are of lower degree $-1$.
All the algebras considered in this paper, are unital
and associative.
Let $A$ be a differential graded algebra. 
Let $M$ be a right $A$-module and $N$ be a left $A$-module.
Denote by $sA$ the suspension of $A$, $(s A)_i=A_{i-1}$.
Let $d_0$ be the differential on the tensor product of complexes
$M\otimes T(sA)\otimes N$.
We denote the tensor product of the elements $m\in M$, $sa_1\in sA$,
\ldots , $sa_k\in sA$ and $n\in N$ by $m[a_1|\cdots|a_k]n$.
Let $d_1$ be the differential on the graded vector space
$M\otimes T(sA)\otimes N$ defined by:
\begin{align*}
d_1m[a_1|\cdots|a_k]n=&(-1)^{\vert m\vert } ma_1[a_2|\cdots|a_k]n\\
&+\sum _{i=1}^{k-1} (-1)^{\varepsilon_i}{m[a_1|\cdots|a_ia_{i+1}|\cdots|a_k]n}\\
&-(-1)^{\varepsilon_{k-1}} m[a_1|\cdots|a_{k-1}]a_kn;
\end{align*}
Here $\varepsilon_i=\vert m\vert +\vert a_1\vert+\cdots +\vert a_i\vert+i$.

The {\it bar construction of $A$ with coefficients in $M$ and in $N$}, denoted
$B(M;A;N)$, is the complex $(M\otimes
T(sA)\otimes N,d_0+d_1)$.
The {\it bar resolution of $A$}, denoted
$B(A;A;A)$, is the differential graded $(A,A)$-bimodule $(A\otimes
T(sA)\otimes A,d_0+d_1)$.
If $A$ is augmented then the {\it reduced bar construction of $A$},
denoted $B(A)$, is $B({\Bbbk};A;{\Bbbk})$.

Denote by $A^{op}$ the opposite algebra of $A$ and by $A^e:=A\otimes A^{op}$
the envelopping algebra of $A$.
Let $M$ be a differential graded $(A,A)$-bimodule.
Recall that any $(A,A)$-bimodule can be considered as a left (or right)
$A^{e}$-module.
The {\it Hochschild chain complex} is the complex
$M\otimes_{A^{e}} B(A;A;A)$ denoted $\mathcal{C}_*(A,M)$.
Explicitly $\mathcal{C}_*(A,M)$ is the complex ($M\otimes T(sA),d_0+d_1)$
with $d_0$ obtained by tensorization and
~\cite[(10) p. 78]{Cuntz-Skandalis-Tamarkin:cyclichomnoncomgeom}
\begin{align*}
d_1m[a_1|\cdots|a_k]=&(-1)^{\vert m\vert } ma_1[a_2|\cdots|a_k]\\
&+\sum _{i=1}^{k-1} (-1)^{\varepsilon_i}{m[a_1|\cdots|a_ia_{i+1}|\cdots|a_k]}\\
&-(-1)^{\vert sa_k\vert\varepsilon_{k-1}} a_km[a_1|\cdots|a_{k-1}].
\end{align*}
The {\it Hochschild homology of $A$ with coefficients in $M$} is the homology $H$
of the Hochschild chain complex:
$$HH_*(A,M):=H(\mathcal{C}_*(A,M)).$$
 The {\it Hochschild cochain complex} of $A$ with coefficients in $M$ is
the   complex $\mbox{Hom}_{A^{e}}(B(A;A;A),M)$ denoted $\mathcal{C}^\ast  (A,M)$.
Explicitly $\mathcal{C}^\ast  (A,M)$ is the complex 
    $$(\mbox{Hom}   (T(sA), M), D_0+D_1).$$
Here for $f \in  \mbox{Hom}( T(sA), M) $, 
  $D_0(f)([\,]) = d_M(f([\,]))$, $D_1(f)([\,])=0$, and for $k\geq 1$,
we have:
   $$D_0(f)([a_1|a_2|...|a_k])
  = d_{M}(f\left([a_1|a_2|...|a_k])\right) - \sum _{i=1} ^k
  (-1)^{\overline \epsilon_i} f([a_1|...|d_Aa_i|...|a_k])$$  and
  $$\renewcommand{\arraystretch}{1.6}
\begin{array}{ll}
D_1(f)([a_1|a_2|...|a_k])= &
   - (-1)^{|sa_1|\, |f|}  a_1 f([a_2|...|a_k])\\ &
- \sum _{i=2} ^k
  (-1)^{\overline \epsilon_i} f([a_1|...|a_{i-1}a_i|...|a_k])\\
 & + (-1) ^{\overline\epsilon _k} f([a_1|a_2|...|a_{k-1}])a_k
  \,,
\end{array}
\renewcommand{\arraystretch}{1}
$$
 where
  $\overline \epsilon _i = |f|+|sa_1|+|sa_2|+...+|sa_{i-1}|$.

The {\it Hochschild cohomology of $A$ with coefficients in $M$} is
 $$
 HH^*(A,M)=  H( \mathcal{C} ^\ast(A,M)).
 $$
Suppose that $A$ has an augmentation $\varepsilon:A\twoheadrightarrow {\Bbbk}$.
Let $\overline{A}:=\text{Ker}\varepsilon$ be the augmentation ideal.
We denote by $\overline{B}(A):=(Ts\overline{A},d_0+d_1)$ the normalized
reduced Bar construction, by
$\overline{\mathcal{C}}_*(A,M):=(M\otimes T(s\overline{A}),d_0+d_1)$ the normalized Hochschild chain complex and by $\overline{\mathcal{C}}^*(A,M):=(\mbox{Hom}   (T(s\overline{A}), M), D_0+D_1)$ the normalized Hochschild cochain complex.

\section{The isomorphism between Hochschild cohomology and Hochschild
  homology for differential graded algebras}\label{iso cohomology homology}
Let $A$ be a differential graded algebra.
Let $P$ and $Q$ be two $A$-bimodules.

The action of $HH^*(A,Q)$ on $HH_*(A,P)$ comes from a (right) action
of the $\mathcal{C}^*(A,Q)$ on $\mathcal{C}_*(A,P)$ given by
~\cite[(18) p. 82]{Cuntz-Skandalis-Tamarkin:cyclichomnoncomgeom},
~\cite{Lambre:dualiteVdB}

$$\cap:\mathcal{C}_*(A,P)\otimes\mathcal{C}^*(A,Q)\rightarrow \mathcal{C}_{*}(A,P\otimes_A Q)$$
\begin{equation}\label{cap Hochschild}
(m[a_1|\dots|a_n],f)\mapsto
(m[a_1|\dots|a_n])\cap f:=\sum_{p=0}^n \pm
(m\otimes_A f[a_1|\dots|a_p])[a_{p+1}|\dots|a_n].
\end{equation}
Here $\pm$ is the Koszul sign
$(-1)^{\vert f\vert (\vert a_1\vert+\dots \vert a_n\vert+n)}$~\cite[Proof of Lemma 16]{Menichi:BV_Hochschild}.

Let $f:A\rightarrow B$ be a morphism of differential graded algebras
and let $N$ be a $B$-bimodule.
The linear map $B\otimes_A N\rightarrow N$, $b\otimes n\mapsto b.n$
is a morphism of $B$-bimodules.
We call again cap product the composite
\begin{equation}\label{cap Hochschild morphism algebres}
\mathcal{C}_*(A,B)\otimes\mathcal{C}^*(A,N)
\buildrel{\cap}\over\rightarrow
\mathcal{C}_{*}(A,B\otimes_A N)
\rightarrow \mathcal{C}_{*}(A,N).
\end{equation}
In this paper, our goal (statement~\ref{iso donne par cap})
is to relate the cap product with $B=A$ to the cap product with
$N=B={\Bbbk}$.
\begin{statement}\label{iso donne par cap}
Let $A$ be an augmented differential graded algebra
such that each $A_i$ is ${\Bbbk}$-free, $i\in\mathbb{Z}$.
Let $c\in HH_d(A,A)$.
Denote by $[m]\in\text{Tor}^A_d({\Bbbk},{\Bbbk})$ the image of $c$ by the morphism
$$
HH_d(A,\varepsilon):HH_d(A,A)\rightarrow HH_d(A,{\Bbbk})=\text{Tor}^A_d({\Bbbk},{\Bbbk}).
$$
Suppose that 

$\bullet$ there exists a positive integer $n$ such that for all
$i\leq -n$ and for all $i\geq n$,
$\text{Tor}^A_i({\Bbbk},{\Bbbk})=0$,

$\bullet$ each
$\text{Tor}^A_i({\Bbbk},{\Bbbk})$
is of finite type, $i\in\mathbb{Z}$,

$\bullet$ the morphism of right $\text{Ext}^*_A({\Bbbk},{\Bbbk})$-modules
$$
Ext^p_A({\Bbbk},{\Bbbk})\buildrel{\cong}\over\rightarrow
\text{Tor}^A_{d-p}({\Bbbk},{\Bbbk}), a\mapsto [m]\cap a
$$
is an isomorphism.

Then for any $A$-bimodule $N$,
 the morphism 

$$
\mathbb{D}^{-1}:HH^p(A,N)\buildrel{\cong}\over\rightarrow
HH_{d-p}(A,N), a\mapsto c\cap a
$$
is also an isomorphism.
\end{statement}
This statement is the Eckmann-Hilton or Koszul dual
of~\cite[Proposition 11]{Menichi:BV_Hochschild}.
In this section, we will prove this statement if $A$ is connected.
But we wonder if this statement is true  in the non-connected case or even for ungraded
algebras.
\begin{propriete}\label{produit tensoriel et dual}
Let $B$ and $N$ be two complexes.
Consider the natural morphism of complexes
$\Theta:B^\vee\otimes N\rightarrow \text{Hom}(B,N)$, which sends $\varphi\otimes n$
to the linear map $f:B\rightarrow N$ defined by $f(b):=\varphi(b)n$.
Suppose that each $B_i$ is ${\Bbbk}$-free.

1) If $B_i=0$ for all $i\leq -n$ and for all $i\geq n$, for some positive integer $n$
and if each  $B_i$ is of finite type or

2) If $H_i(B)=0$ for all $i\leq -n$ and for all $i\geq n$, for some positive integer $n$
and if each  $H_i(B)$ is of finite type

\noindent then $\Theta$ is a homotopy equivalence.
\end{propriete}
\begin{proof}
1) Since $B$ is bounded, the component of degre $n$ of
$\text{Hom}(B,N)$ is the direct sum $\oplus_{q\in\mathbb{Z}}\text{Hom}(B_{q-n},N_q)$.
Since $B_{q-n}$ is free of finite type, $\text{Hom}(B_{q-n},N_q)$ is isomorphic
to $B_{q-n}^\vee\otimes N_q$. Therefore $\Theta$ is an isomorphism.

\noindent 2) Since ${\Bbbk}$ is a principal ideal domain,
the proof of~\cite[Lemma 5.5.9]{Spanier:livre} shows that
there exists an complex $B'$ satisfying 1) homotopy equivalent to $B$.
By naturality of $\Theta$, $\Theta$ is a homotopy equivalence of complexes.
\end{proof}
\begin{lem}\label{cas du bimodule trivial}
The statement holds whenever $N$ is a trivial $A$-bimodule, i.e. 
$a.n=\varepsilon(a)n=n.a$ for $a\in A$ and $n\in N$.
\end{lem}
\begin{proof}
Since $N$ is a trivial $A$-bimodule, the normalized Hochschild chain complex
$\overline{\mathcal{C}}_*(A,N)$ is just the tensor product of complexes
$\overline{\mathcal{C}}_*(A,{\Bbbk})\otimes N=\overline{B}(A)\otimes
N$
(This is also true for the unnormalized Hochschild chain complex, but
it is less obvious).
And the normalized Hochschild cochain complex
$\overline{\mathcal{C}}^*(A,N)$ is just the Hom complex
$\text{Hom}(\overline{\mathcal{C}}_*(A,{\Bbbk}),N)
=\text{Hom}(\overline{B}(A),N)$.
Since the augmentation ideal of $A$, $\overline{A}$, is ${\Bbbk}$-free,
$\overline{B}(A)$ is also  ${\Bbbk}$-free.
Each $H_i(\overline{B}(A))$ is of finite type and
$H_i(\overline{B}(A))=\text{Tor}_i^A({\Bbbk},{\Bbbk})$
is null if $i\leq -n $ or $i\geq n$.
Therefore by part 2) of Property~\ref{produit tensoriel et dual},
$\Theta:\overline{B}(A)^\vee\otimes N\buildrel{\simeq}\over\rightarrow \text{Hom}(\overline{B}(A),N)$
is a quasi-isomorphism.
A straightforward calculation shows that the following diagram commutes

$$
\xymatrix{
\overline{B}(A)^\vee\otimes N\ar[r]^-\Theta_-\simeq\ar[dr]_{([m]\cap -)\otimes N}
& \text{Hom}(\overline{B}(A),N)=\overline{\mathcal{C}}^*(A,N)\ar[d]^{c\cap -}\\
& \overline{B}(A)\otimes N=\overline{\mathcal{C}}_*(A,N)
}
$$
Since $\overline{B}(A)$ is ${\Bbbk}$-free and its dual $\overline{B}(A)^\vee$
is torsion free,
by naturality of Kunneth formula~\cite[Theorem 5.3.3]{Spanier:livre}, $([m]\cap -)\otimes N$ is a quasi-isomorphism.
Therefore $c\cap -$ is also a quasi-isomorphism.
\end{proof}
\begin{proposition}\label{quasi-iso du cap apres limite inverse}
Let $A$ be an augmented differential graded algebra.
Let $N$ be an $A$-bimodule. And let $c\in HH_d(A,A)$ satisfying the
hypotheses of Statement~\ref{iso donne par cap}.
For any $k\geq 0$,
let $F^k:=\overline{A^e}^k.N$.
Then taking the inverse limit of the cap product with $c$ induces
a quasi-isomorphism of complexes
$$
\lim_\leftarrow c\cap -:\lim_\leftarrow \mathcal{C}^*(A,N/F^k)\buildrel{\simeq}\over\rightarrow \lim_\leftarrow
\mathcal{C}_*(A,N/F^k).
$$
\end{proposition}
\begin{proof}
Consider the augmentation ideal $\overline{A^e}$ of the envelopping algebra
$A^e$.
For any $k\geq 0$, let $\overline{A^e}^k$ be the image of the iterated tensor
product $\overline{A^e}^{\otimes k}$ by the iterated multiplication of
$A^e$, $\mu: (A^e)^{\otimes k}\rightarrow A^e$ and let
$F^k$ be the image of $\overline{A^e}^k\otimes N$ by the action
$A^e\otimes N\rightarrow N$.

The $F^k$ form a decreasing filtration of sub-$A$-bimodules and sub-complexes of $N$.
Since $F^k/F^{k+1}$ is a trivial $A$-bimodule, by Lemma~\ref{cas du bimodule trivial},
the morphism of complexes
$$
\mathcal{C}^*(A,F^k/F^{k+1})\buildrel{\simeq}\over\rightarrow
\mathcal{C}_*(A,F^k/F^{k+1}), a\mapsto c\cap a
$$
is a quasi-isomorphism.
By Noether theorem, we have the short exact sequences of
$A$-bimodules

$$0\rightarrow F^k/F^{k+1}\rightarrow N/F^{k+1} \rightarrow N/F^{k}\rightarrow 0.
$$
Since $TsA$ is ${\Bbbk}$-free, the functors $\text{Hom}_{\Bbbk}(TsA,-)$
and $-\otimes_{\Bbbk}TsA$ preserve short exact sequences.
Therefore consider the morphism of short exact sequences of complexes induced by the cap product with $c$
$$
\xymatrix{
0\ar[r]
&\mathcal{C}^*(A,F^k/F^{k+1})\ar[r]\ar[d]^{c\cap -}_{\simeq}
&\mathcal{C}^*(A,N/F^{k+1})\ar[r]\ar[d]^{c\cap -}
& \mathcal{C}^*(A,N/F^{k})\ar[r]\ar[d]^{c\cap -}
&0\\
0\ar[r]
&\mathcal{C}_*(A,F^k/F^{k+1})\ar[r]
&\mathcal{C}_*(A,N/F^{k+1})\ar[r]
& \mathcal{C}_*(A,N/F^{k})\ar[r]
&0
}
$$
Using the long exact sequences associated and the five lemma, by induction on $k$,
we obtain that the morphism of complexes
$$
\mathcal{C}^*(A,N/F^{k})\buildrel{\simeq}\over\rightarrow
\mathcal{C}_*(A,N/F^{k}), a\mapsto c\cap a
$$
is a quasi-isomorphism for all $k\geq 0$.

The two towers of complexes

$$
\cdots\twoheadrightarrow  \mathcal{C}^*(A,N/F^{k+1})\twoheadrightarrow \mathcal{C}^*(A,N/F^{k})\twoheadrightarrow\cdots
$$
$$
\cdots\twoheadrightarrow  \mathcal{C}_*(A,N/F^{k+1})\twoheadrightarrow \mathcal{C}_*(A,N/F^{k})\twoheadrightarrow\cdots
$$
satisfy the trivial Mittag-Leffler condition, since all the maps in the two towers
are onto. Therefore by naturality of~\cite[Theorem 3.5.8]{Weibel:inthomalg},
for each $p\in\mathbb{Z}$, we have the 
 morphism of short exact sequences induced by the cap product with $c$
$$
\xymatrix{
\displaystyle{\lim_\leftarrow}^1 HH^{p-1}(A,N/F^{k})\ar[r]\ar[d]^{\displaystyle{\lim_\leftarrow}^1 c\cap -}_\cong
&H^{p}\displaystyle{\lim_\leftarrow} \mathcal{C}^*(A,N/F^{k})\ar[r]\ar[d]^{H(\displaystyle\lim_\leftarrow c\cap -)}
& \displaystyle{\lim_\leftarrow} HH^p(A,N/F^{k})\ar[d]^{\displaystyle\lim_\leftarrow c\cap -}_\cong\\
\displaystyle{\lim_\leftarrow}^1 HH_{d+1-p}(A,N/F^{k})\ar[r]
&H_{d-p}\displaystyle{\lim_\leftarrow}\mathcal{C}_*(A,N/F^{k})\ar[r]
& \displaystyle{\lim_\leftarrow} HH_{d-p}(A,N/F^{k})
}
$$
Using the five Lemma again, we obtain that the middle morphism

$$
H(\lim_\leftarrow c\cap -):
H^{p}\lim_\leftarrow \mathcal{C}^*(A,N/F^{k})\rightarrow
H_{d-p}\lim_\leftarrow \mathcal{C}_*(A,N/F^{k})
$$
is an isomorphism.
\end{proof}
\begin{cor}\label{cas simplement connexe}
The statement is true if $A$ and $N$ are non-negatively lower graded and
$H_0(\varepsilon):H_0(A)\buildrel{\cong}\over\rightarrow {\Bbbk}$
is an isomorphism.
\end{cor}
\begin{proof}
Case 1: We first suppose that 
$\varepsilon:A_0\buildrel{\cong}\over\rightarrow {\Bbbk}
$ is an isomorphism.
Then $\overline{A^e}^k$
is concentrated in degres $\geq k$. Therefore
 $F^k$ and $\mathcal{C}_*(A,F^{k})$ are also concentrated in degres $\geq k$.
This means that for $n<k$ their components of degre $n$, $(F^k)_n$ and
$[\mathcal{C}_*(A,F^{k})]_n$, are trivial.
Therefore the tower in degre $n$
$$\cdots\rightarrow (N/F^{k+1})_n\twoheadrightarrow (N/F^{k})_n\rightarrow\cdots$$
is constant and equal to $N_n$ for $k>n$.
This implies that  $N_n=\displaystyle{\lim_\leftarrow} (N/F^{k})_n$.
Therefore as complexes and as $A$-bimodule, $N=\displaystyle{\lim_\leftarrow} N/F^{k}$.

Since $\mathcal{C}_*(A,N/F^{k})$ is the quotient $\mathcal{C}_*(A,N)/\mathcal{C}_*(A,F^{k})$,
we also have that as complexes,
$$\mathcal{C}_*(A,N)=\lim_\leftarrow \mathcal{C}_*(A,N/F^{k})$$

The functor $\mathcal{C}^*(A,-)$ from (differential) $A$-bimodules to complexes is a right adjoint (to the functor $B(A;A;A)\otimes -$).
Therefore $\mathcal{C}^*(A,-)$ preserves inverse limits.
 Since $N=\displaystyle{\lim_\leftarrow} N/F^{k}$ in the category of (differential) $A$-bimodules,
we obtain that as complex 
$$\mathcal{C}^*(A,N)=\mathcal{C}^*(A,\displaystyle{\lim_\leftarrow} N/F^{k})=\lim_\leftarrow \mathcal{C}^*(A,N/F^{k}).$$
Since for any $k\geq 0$, the following square commutes
$$
\xymatrix{
\mathcal{C}^*(A,N)\ar[r]\ar[d]^{c\cap -}
&\mathcal{C}^*(A,N/F^{k})\ar[d]^{c\cap -}\\
\mathcal{C}_*(A,N)\ar[r]
&\mathcal{C}_*(A,N/F^{k})},
$$
the quasi-isomorphism given by Proposition~\ref{quasi-iso du cap apres limite inverse}
$$
\lim_\leftarrow c\cap -:\lim_\leftarrow \mathcal{C}^*(A,N/F^k)\buildrel{\simeq}\over\rightarrow \lim_\leftarrow
\mathcal{C}_*(A,N/F^k)
$$
coincides with
 $c\cap -:\mathcal{C}^*(A,N)   \rightarrow \mathcal{C}_*(A,N)$.

Case 2: Now, we only suppose that
$H_0(\varepsilon):H_0(A)\buildrel{\cong}\over\rightarrow{\Bbbk}$
is an isomorphism. Let $\tilde{A}$ be the graded ${\Bbbk}$-module
defined by $\tilde{A}_0={\Bbbk}$,
$\tilde{A}_1=\text{Ker}\;d:A_1\rightarrow A_0$, $\tilde{A}_n= A_n$ for $n\geq 2$
(Compare with the upper graded version in~\cite[p. 184]{Felix-Halperin-Thomas:ratht}).
Clearly $\tilde{A}$ is a ${\Bbbk}$-free
sub differential graded algebra of $A$ and the
inclusion $j:\tilde{A}\buildrel{\simeq}\over\hookrightarrow A$
is a quasi-isomorphism since $\text{Im}\;(d:A_1\rightarrow A_0)$ is equal to $\overline{A}_0$.

Since
the augmentation ideals of $A$ and $\tilde{A}$, $\overline{A}$ and $\overline{\tilde{A}}$,
are ${\Bbbk}$-free and non-negatively lower graded, by~\cite[5.3.5]{LodayJ.:cych} or~\cite[4.3(iii)]{Felix-Halperin-Thomas:dgait},
the three morphisms $HH_*(j,N):HH_*(\tilde{A},N)\buildrel{\cong}\over\rightarrow HH_*(A,N)$,
 $HH^*(j,N):HH^*(A,N)\buildrel{\cong}\over\rightarrow HH^*(\tilde{A},N)$ and
$HH_*(j,j):HH_*(\tilde{A},\tilde{A})\buildrel{\cong}\over\rightarrow HH_*(A,A)$ are all isomorphims.
Let $\tilde{c}\in HH_d(\tilde{A},\tilde{A})$ such that
$HH_d(j,j)(\tilde{c})=c$. Using the definition of the cap product, it is
straightforward to check that the following square commutes
$$
\xymatrix{
HH^*(A,N)\ar[r]^{HH^*(j,N)}_\cong\ar[d]^{c\cap -}
&HH^*(\tilde{A},N)\ar[d]^{\tilde{c}\cap -}\\
HH_*(A,N)
&HH_*(\tilde{A},N)\ar[l]^{HH_*(j,N)}_\cong
}
$$
Let $\tilde{[m]}\in\text{Tor}^{\tilde{A}}_d({\Bbbk},{\Bbbk})$
such that $\text{Tor}^{j}_d({\Bbbk},{\Bbbk})(\tilde{[m]})=[m]$.
When $N={\Bbbk}$, the previous square specializes to the following commutative
square

$$
\xymatrix{
\text{Ext}^*_A({\Bbbk},{\Bbbk})\ar[r]^{\text{Ext}^*_j({\Bbbk},{\Bbbk})}_\cong\ar[d]_{[m]\cap -}^\cong
&\text{Ext}^*_{\tilde{A}}({\Bbbk},{\Bbbk})\ar[d]^{\tilde{[m]}\cap -}\\
\text{Tor}^A_*({\Bbbk},{\Bbbk})
&\text{Tor}_*^{\tilde{A}}({\Bbbk},{\Bbbk})\ar[l]^{\text{Tor}_*^j({\Bbbk},{\Bbbk})}_\cong
}
$$
By hypothesis, $[m]\cap -$ is an isomorphism. Therefore $\tilde{[m]}\cap -$
is also an isomorphism. So since $\tilde{A}_0={\Bbbk}$, we have seen 
in the case 1, that

$$
\tilde{c}\cap-:HH^*(\tilde{A},N)\buildrel{\cong}\over\rightarrow HH_*(\tilde{A},N)
$$
is an isomorphism. Therefore
$$
c\cap-:HH^*(A,N)\buildrel{\cong}\over\rightarrow HH_*(A,N)
$$
is also an isomorphism.
\end{proof}
\begin{cor}\label{cas cochain connexe}
The statement is true if $A$ and $N$ are non-negatively upper graded,
$H^0(\varepsilon):H^0(A)\buildrel{\cong}\over\rightarrow {\Bbbk}$
is an isomorphism and ${\Bbbk}$ is a field.
\end{cor}
\begin{proof}
Case 1: We first suppose that
$\varepsilon:A^0\buildrel{\cong}\over\rightarrow {\Bbbk}$ is an isomorphism.
Since $T(sA)$ has non-trivial elements of negative degres,
we need to use the normalized Hochschild chain and cochain complexes
$\overline{\mathcal{C}}_*$ and $\overline{\mathcal{C}}^*$ instead of the unnormalized
$\mathcal{C}_*$ and $\mathcal{C}^*$. Now the proof is the same as in Case 1
of the proof of Corollary~\ref{cas simplement connexe}.

Case 2: Now, we only suppose that that
$H^0(\varepsilon):H^0(A)\buildrel{\cong}\over\rightarrow {\Bbbk}$
is an isomorphism. Since ${\Bbbk}$ is a field, by~\cite[p. 184]{Felix-Halperin-Thomas:ratht}), there exists a differential graded
algebra  $\tilde{A}$, non-negatively upper graded,
equipped with a quasi-isomorphism
$j:\tilde{A}\buildrel{\simeq}\over\rightarrow A$ such that
$\tilde{A}^0={\Bbbk}$.
Now the rest of the proof is exactly the same as in Case 2
of the proof of Corollary~\ref{cas simplement connexe}.
\end{proof}
\section{Comparison of the Cap products in Hochschild and group (co)homology}
Let $G$ be a discrete group.
Let $M$ and $N$ be two ${\Bbbk}[G]$-bimodules.
Let $\eta:{\Bbbk}\rightarrow {\Bbbk}[G]$ be the unit map.
Let $E: {\Bbbk}[G]\rightarrow {\Bbbk}[G\times G^{op}]$ be
the morphism of algebras mapping $g$ to $(g,g^{-1})$.
Let $$\tilde{\eta}: {\Bbbk}[G\times G^{op}]\otimes_{{\Bbbk}[G]}{\Bbbk}
\rightarrow {\Bbbk}[G]$$
be the unique morphism of left ${\Bbbk}[G\times G^{op}]$-modules
extending $\eta$. Since ${\Bbbk}[G\times G^{op}]$ is flat as
left $ {\Bbbk}[G]$-module via $E$ and since $\tilde{\eta}$
is an isomorphism, by Eckmann-Schapiro~\cite[Chapt IV.Proposition 12.2]{Hilton-Stammbach}, we obtain the well-known isomorphisms between Hochschild (co)homology
and group (co)homology:
$$
\text{Ext}_E^*(\eta,N):HH^*({\Bbbk}[G],N)=
\text{Ext}_{{\Bbbk}[G\times G^{op}]}^*({\Bbbk}[G],N)
\buildrel{\cong}\over\rightarrow
\text{Ext}_{{\Bbbk}[G]}^*({\Bbbk},\tilde{N})=H^*(G,\tilde{N}).
$$
and
$$\text{Tor}^E_*(M,\eta):
H_*(G,\tilde{M})=\text{Tor}^{{\Bbbk}[G]}_*(\tilde{M},{\Bbbk})
\buildrel{\cong}\over\rightarrow\text{Tor}_{{\Bbbk}[G\times G^{op}]}^*(M,{\Bbbk}[G])=HH_*({\Bbbk}[G],M).
$$
Here $\tilde{M}$ and $\tilde{N}$ denote the ${\Bbbk}[G]$-modules
obtained by restriction of scalar via $E$.
Note that we regard any left ${\Bbbk}[G]$-module as an right ${\Bbbk}[G]$-module
via $g\mapsto g^{-1}$~\cite[p. 55]{Brown:cohgro}.
\begin{proposition}\label{comparaison cup cap Hochschild group cohomology}
Remark that the canonical surjection
$$q:\tilde{M}\otimes\tilde{N}\twoheadrightarrow \widetilde{M\otimes_{{\Bbbk}[G]}N}$$
is a morphism of ${\Bbbk}[G]$-modules, since $q(gmg^{-1}\otimes gng^{-1})=gm\otimes ng^{-1}$.

i) Cup product $\cup$ in Hochschild cohomology versus cup product in group
cohomology (slight extension of~\cite[Proposition 3.1]{Siegel-Witherspoon:Hochschildcoh}).
The following diagram commutes

\xymatrix{
HH^*({\Bbbk}[G],M)\otimes HH^*({\Bbbk}[G],N)
\ar[rr]^\cup \ar[d]|{\text{Ext}_E^*(\eta,M)\otimes\text{Ext}_E^*(\eta,N)}
&& HH^*({\Bbbk}[G],M\otimes_{{\Bbbk}[G]}N)
\ar[d]_{\text{Ext}_E^*(\eta,M\otimes_{{\Bbbk}[G]}N)}\\
H^*(G,\tilde{M})\otimes H^*(G,\tilde{N})\ar[r]_\cup
& H^*(G,\tilde{M}\otimes\tilde{N})\ar[r]_{H^*(G,q)}
&H^*(G,\widetilde{M\otimes_{{\Bbbk}[G]}N})
}

ii) Cap products $\cap$.
The following diagram commutes
$$
\xymatrix{
HH_*({\Bbbk}[G],M)\otimes HH^*({\Bbbk}[G],N)\ar[rr]^\cap
&& HH_*({\Bbbk}[G],M\otimes_{{\Bbbk}[G]}N)\\
H_*(G,\tilde{M})\otimes H^*(G,\tilde{N})\ar[r]_\cap
\ar[u]|{\text{Tor}^E_*(M,\eta)\otimes\text{Ext}_E^*(\eta,N)^{-1}}
& H_*(G,\tilde{M}\otimes\tilde{N})\ar[r]_{H_*(G,q)}
&H_*(G,\widetilde{M\otimes_{{\Bbbk}[G]}N})
\ar[u]^{\text{Tor}^E_*(M\otimes_{{\Bbbk}[G]}N,\eta)}
}
$$
\end{proposition}
\begin{rem} 
In the case $N={\Bbbk}[G]$~\cite[(3.3)]{Siegel-Witherspoon:Hochschildcoh}, the morphism of  ${\Bbbk}[G]$-modules
$q:\tilde{M}\otimes\tilde{{\Bbbk}[G]}\twoheadrightarrow \widetilde{M\otimes_{{\Bbbk}[G]}{\Bbbk}[G]}\cong\tilde{M}$ is simply the action $m\otimes g\mapsto m.g$.

In the case $M=N={\Bbbk}[G]$, the diagram i) in
Proposition~\ref{comparaison cup cap Hochschild group cohomology}
means that
$$
\text{Ext}_E^*(\eta,{\Bbbk}[G]):HH^*({\Bbbk}[G],{\Bbbk}[G])
\rightarrow H^*(G,\tilde{{\Bbbk}[G]})
$$
is a morphism of graded algebras.

In the case $N={\Bbbk}[G]$, the diagram ii) means that
$$\text{Tor}^E_*(M,\eta):H_*(G,\tilde{M})\rightarrow
HH_*({\Bbbk}[G],M)
$$
is a morphism of right $HH^*({\Bbbk}[G],{\Bbbk}[G])$-modules:
$$
\text{Tor}^E_*(\eta,{\Bbbk}[G])\left(\alpha\cap
\text{Ext}_E^*(\eta,{\Bbbk}[G])(\varphi)\right)=
\text{Tor}^E_*(\eta,{\Bbbk}[G])(\alpha)\cap\varphi
$$
for any $\alpha\in H_*(G,\tilde{M})$ and any
$\varphi\in HH^*({\Bbbk}[G],{\Bbbk}[G])$.
\end{rem}
\begin{proof}
Siegel and Witherspoon~\cite[Proposition 3.1]{Siegel-Witherspoon:Hochschildcoh}
proved i) using that for any projective resolution $P$ of ${\Bbbk}$
as left ${\Bbbk}[G]$-modules,
$$X:={\Bbbk}[G\times G^{op}]\otimes_{{\Bbbk}[G]}P$$ is a projective
resolution of  ${\Bbbk}[G]$ as ${\Bbbk}[G]$-bimodules.
Let $\iota:P\hookrightarrow\tilde{X}$ the left ${\Bbbk}[G]$-linear map
defined by $\iota(x)=(1,1)\otimes x$.
Using that
$$\text{Hom}_E(\iota,N):\text{Hom}_{{\Bbbk}[G\times G^{op}]}(X,N)
\buildrel{\cong}\over\rightarrow \text{Hom}_{{\Bbbk}[G]}(P,\tilde{N})$$
is an isomorphism of complexes inducing $\text{Ext}_E^*(\eta,N)$
and that 
$$
M\otimes_E\iota:\tilde{M}\otimes_{{\Bbbk}[G]}P
\buildrel{\cong}\over\rightarrow
M\otimes_{{\Bbbk}[G\times G^{op}]}X
$$ 
is an isomorphism of complexes inducing $\text{Tor}_*^E(M,\eta)$,
Siegel and Witherspoon~\cite[Proposition 3.1]{Siegel-Witherspoon:Hochschildcoh}
proved i). But one can also prove similarly ii).

We find more simple to give a proof of ii) using the Bar resolution.
Let $\iota:B({\Bbbk}[G];{\Bbbk}[G];{\Bbbk})\rightarrow 
\widetilde{B({\Bbbk}[G];{\Bbbk}[G];{\Bbbk}[G])}$
be the linear map defined by
$$\iota(g_0[g_1|\cdots|g_n])=g_0[g_1|\cdots|g_n]g_n^{-1}\dots g_0^{-1}.$$
Obviously $\iota$ fits into the commutative diagram of
left ${\Bbbk[G]}$-modules
$$
\xymatrix{
\widetilde{B({\Bbbk}[G];{\Bbbk}[G];{\Bbbk}[G])}\ar[r]
& {\Bbbk}[G]\\
B({\Bbbk}[G];{\Bbbk}[G];{\Bbbk})\ar[r]\ar@{.>}[u]^{\iota}
&{\Bbbk}\ar[u]_\eta
}$$
A straightforward computation shows that $\iota$ is a morphism
of complexes.
Therefore
$\text{Hom}_E(\iota,N)$ is an morphism of complexes from
$\mathcal{C}^*({\Bbbk}[G],N)\cong(\text{Hom}_{{\Bbbk}[G\times G^{op}]}(B({\Bbbk}[G];{\Bbbk}[G];{\Bbbk}[G]),N)$ to
$\text{Hom}_{{\Bbbk}[G]}(B({\Bbbk}[G];{\Bbbk}[G];{\Bbbk}),\tilde{N})$.
 inducing $\text{Ext}_E^*(\eta,N)$
and
$M\otimes_E\iota$ is an morphism of complexes
 from $$B(\tilde{M};{\Bbbk}[G];{\Bbbk})\cong\tilde{M}\otimes_{{\Bbbk}[G]}B({\Bbbk}[G];{\Bbbk}[G];{\Bbbk})$$
to
$$M\otimes_{{\Bbbk}[G\times G^{op}]}B({\Bbbk}[G];{\Bbbk}[G];{\Bbbk}[G])
\cong \mathcal{C}_*({\Bbbk}[G],M)
$$ 
 inducing $\text{Tor}_*^E(M,\eta)$.
Explicitly $M\otimes_E\iota$ is the morphism of complexes
$$
B(\tilde{M};{\Bbbk}[G];{\Bbbk})\rightarrow \mathcal{C}_*({\Bbbk}[G],M)
$$
defined by~\cite[(2.20)]{Feng-Tsygan}
\begin{equation}\label{application de Feng-Tsygan}
\xi(m[g_1|\dots|g_n]= g_n^{-1}\dots g_1^{-1}m[g_1|\dots|g_n].
\end{equation}
And $\text{Hom}_E(\iota,N):\mathcal{C}^*({\Bbbk}[G],N)\rightarrow
\text{Hom}(B({\Bbbk}[G]),\tilde{N}),d$ is the morphism of
complexes $\xi$ mapping $\varphi\in\mathcal{C}^*({\Bbbk}[G],N)$
to the linear map $\xi(\varphi):{\Bbbk}[G]\rightarrow\tilde{N}$ defined
by 
$$
\xi(\varphi)([g_1|\dots|g_n])=\varphi([g_1|\dots|g_n])g_n^{-1}\dots g_1^{-1}.
$$
Both $M\otimes_E\iota$ and $\text{Hom}_E(\iota,N)$ are in fact isomorphisms
of complexes.
The inverse of $M\otimes_E\iota$ is the morphism of complexes
$\Phi:\mathcal{C}_*({\Bbbk}[G],M)\rightarrow B(\tilde{M};{\Bbbk}[G];{\Bbbk})$
defined by~\cite[7.4.2.1]{LodayJ.:cych}
$$
\Phi(m[g_1|\dots|g_n])=g_1\dots g_n m[g_1|\dots|g_n].
$$
Let $F$ be any projective resolution of ${\Bbbk}$ as left
${\Bbbk}[G]$-module. Let $P$ and $Q$ be two ${\Bbbk}[G]$-modules.
The cap product in group cohomology is the composite~\cite[p. 113]{Brown:cohgro},
denoted $\cap$
$$
\xymatrix{
P\otimes_{{\Bbbk}[G]}F\otimes\text{Hom}_{{\Bbbk}[G]}(F,Q)
\ar[d]|{Id\otimes_{{\Bbbk}[G]}\Delta\otimes_{{\Bbbk}[G]} Id}\\
P\otimes_{{\Bbbk}[G]}(F\otimes F)\otimes\text{Hom}_{{\Bbbk}[G]}(F,Q)\ar[d]^\gamma\\
(P\otimes Q)\otimes_{{\Bbbk}[G]}F
}$$
where $\gamma(a\otimes x\otimes y\otimes u)=
(-1)^{\vert u\vert \vert x\vert+\vert u\vert\vert y\vert}a\otimes u(x)\otimes y$
and $\Delta$ is a {\it diagonal approximation}.
In the case, $F$ is the Bar resolution $B({\Bbbk}[G];{\Bbbk}[G];{\Bbbk})$, one can take
$\Delta$ to be the Alexander-Whitney map
$$AW:B({\Bbbk}[G];{\Bbbk}[G];{\Bbbk})\rightarrow
B({\Bbbk}[G];{\Bbbk}[G];{\Bbbk})\otimes B({\Bbbk}[G];{\Bbbk}[G];{\Bbbk})$$
defined by~\cite[p. 108 (1.4)]{Brown:cohgro}:
$$
AW(g_0[g_1|\dots|g_n])=\sum_{p=0}^n g_0[g_1|\dots|g_p]\otimes g_0\dots g_p[g_{p+1}|\dots|g_n].
$$
Therefore the cap product
$$\cap:B(P;{\Bbbk}[G];{\Bbbk})\otimes\text{Hom}(B({\Bbbk}[G]),Q),d\rightarrow
B(P\otimes Q;{\Bbbk}[G];{\Bbbk})$$ is the morphism of complexes
mapping $m[g_1|\dots|g_n]\otimes u:G^p\rightarrow Q$
to $m.g_1\dots g_p\otimes u(g_1,\dots,g_p).g_1\dots g_p[g_{p+1}|\dots|g_n]$.
Using the explicit formula~(\ref{cap Hochschild}) for the cap product in Hochschild cohomology,
it is easy to check that the following diagram commutes
$$
\xymatrix{
\mathcal{C}_*({\Bbbk}[G],M)\otimes \mathcal{C}^*({\Bbbk}[G],N)\ar[rr]^\cap\ar[d]|{\Phi\otimes\text{Hom}_E(\iota,N)}
&& \mathcal{C}_*({\Bbbk}[G],M\otimes_{{\Bbbk}[G]}N)\ar[d]^\Phi\\
B(\tilde{M};{\Bbbk}[G];{\Bbbk})\otimes B(\tilde{N};{\Bbbk}[G];{\Bbbk})\ar[r]_-\cap
& B(\tilde{M}\otimes\tilde{N};{\Bbbk}[G];{\Bbbk})\ar[r]_{B(q;{\Bbbk}[G];{\Bbbk})}
&B(\widetilde{M\otimes_{{\Bbbk}[G]}N};{\Bbbk}[G];{\Bbbk})
}$$
By applying homology, ii) is proved.
\end{proof} 
\begin{defin}~\cite[7.4.5 when z=1]{LodayJ.:cych}\label{definition de la section}
Let $\sigma:B({\Bbbk}[G])\hookrightarrow \mathcal{C}_*({\Bbbk}[G],{\Bbbk}[G])$
be the linear map defined by
$$
\sigma([g_1|\dots|g_n])=g_n^{-1}\dots g_1^{-1}[g_1|\dots|g_n].
$$
\end{defin}
\begin{propriete}\label{proprietes section}

i) ~\cite[7.4.5 when z=1]{LodayJ.:cych} The map $\sigma$ is a morphism
of cyclic modules.

ii) The morphism of complexes $\sigma$ coincides with the composite

$$
B({\Bbbk}[G])\buildrel{B(\eta;{\Bbbk}[G];{\Bbbk})}\over\rightarrow
B(\widetilde{{\Bbbk}[G]};{\Bbbk}[G];{\Bbbk})\build\rightarrow_\cong^\xi
\mathcal{C}_*({\Bbbk}[G];{\Bbbk}[G]).
$$
Here $\xi$ is the isomorphism of complexes defined by~(\ref{application de Feng-Tsygan}). Note that the unit map $\eta:{\Bbbk}\rightarrow\widetilde{{\Bbbk}[G]}$
is a morphism of ${\Bbbk}[G]$-modules.

iii) In particular, in homology, $\sigma$ coincides with
$$\text{Tor}^E(\eta,\eta):H_*(G;{\Bbbk})\rightarrow HH_*({\Bbbk}[G];{\Bbbk}[G]).$$

iv) The map $\sigma$ is a section of
$$\mathcal{C}_*({\Bbbk}[G],\varepsilon):\mathcal{C}_*({\Bbbk}[G],{\Bbbk}[G])
\rightarrow \mathcal{C}_*({\Bbbk}[G],{\Bbbk})=B({\Bbbk}[G]).$$
\end{propriete}
\begin{cor}\label{cap group egal cap Hochschild avec section}
Let $G$ be any discrete group. Let $N$ be a ${\Bbbk}[G]$-bimodule.
Let $\sigma:H_*(G;{\Bbbk})\rightarrow HH_*({\Bbbk}[G];{\Bbbk}[G])$ be the
section of $HH_*({\Bbbk}[G],\varepsilon):HH_*({\Bbbk}[G],{\Bbbk}[G])
\rightarrow H_*(G,{\Bbbk})$ defined in
Definition~\ref{definition de la section}.
Let $z\in H_d(G,{\Bbbk})$ be any element in group homology.
Then the following square commutes
$$
\xymatrix{
H^p(G,\tilde{N})\ar[r]^{z\cap -}
&H_{d-p}(G,\tilde{N})\ar[d]^{\text{Tor}^E_*(N,\eta)}_\cong\\
HH^p({\Bbbk}[G],N)\ar[r]^{\sigma(z)\cap -}\ar[u]^{\text{Ext}_E^*(\eta,N)}_\cong
& HH_{d-p}({\Bbbk}[G],N)
}
$$
\end{cor}
\begin{proof}
$$
\xymatrix{
HH_*({\Bbbk}[G],{\Bbbk}[G])\otimes HH^*({\Bbbk}[G],N)\ar[rr]^\cap
&& HH_*({\Bbbk}[G],{\Bbbk}[G]\otimes_{{\Bbbk}[G]}N)\\
H_*(G,\widetilde{{\Bbbk}[G]})\otimes H^*(G,\tilde{N})\ar[r]_\cap
\ar[u]|{\text{Tor}^E_*({\Bbbk}[G],\eta)\otimes\text{Ext}_E^*(\eta,N)^{-1}}
& H_*(G,\widetilde{{\Bbbk}[G]}\otimes\tilde{N})\ar[r]_{H_*(G,q)}
&H_*(G,\tilde{N})
\ar[u]^{\text{Tor}^E_*(N,\eta)}\\
H_*(G,{\Bbbk})\otimes H^*(G,\tilde{N})\ar[r]_\cap
\ar[u]|{H_*(G,\eta)\otimes Id}
& H_*(G,{\Bbbk}\otimes\tilde{N})\ar[u]^{H_*(G,\eta\otimes\tilde{N} )}\ar[ur]_\cong
}
$$
The top rectangle commutes by ii) of Proposition~\ref{comparaison cup cap Hochschild group cohomology} in the case $M={\Bbbk}[G]$.
The bottom square commutes by naturality of the cap product in group
(co)homology with respect to the morphism of ${\Bbbk}[G]$-modules
$\eta:{\Bbbk}\rightarrow\widetilde{{\Bbbk}[G]}$.
The bottom triangle commutes by functoriality of $H_*(G,-)$.
By ii) or iii) of Property~\ref{proprietes section}, the vertical composite is
$$\sigma\otimes\text{Ext}_E^*(\eta,N)^{-1}:
H_*(G,{\Bbbk})\otimes H^*(G,\tilde{N})\rightarrow
HH_*({\Bbbk}[G],{\Bbbk}[G])\otimes HH^*({\Bbbk}[G],N).
$$
\end{proof}
\section{A new definition of Batalin-Vilkovisky algebras}
\begin{defin}\label{definition algebre de Gerstenhaber}
A {\it Gerstenhaber algebra} is a
commutative graded algebra $A$
equipped with a linear map
$\{-,-\}:A_i \otimes A_j \to A_{i+j+1}$ of degree $1$
such that:

\noindent a) the bracket $\{-,-\}$ gives $A$ a structure of graded
Lie algebra of degree $1$. This means that for each $a$, $b$ and $c\in A$
\begin{equation}\label{antisymmetrie}
\{a,b\}=-(-1)^{(\vert a\vert+1)(\vert b\vert+1)}\{b,a\}\text{ and} 
\end{equation}
$$\{a,\{b,c\}\}=\{\{a,b\},c\}+(-1)^{(\vert a\vert+1)(\vert b\vert+1)}
\{b,\{a,c\}\}.$$

\noindent b)  the product and the Lie bracket satisfy the following relation
called the Poisson relation:
$$\{a,bc\}=\{a,b\}c+(-1)^{(\vert a\vert+1)\vert b\vert}b\{a,c\}.$$
\end{defin}
\begin{defin}\label{definition BV algebre}
A {\it Batalin-Vilkovisky algebra} is a Gerstenhaber algebra
$A$ equipped with a degree $1$ linear map $\Delta:A_{i}\rightarrow A_{i+1}$
such that $\Delta\circ\Delta=0$ and such that the bracket is given by
\begin{equation}\label{relation BV} 
\{a,b\}=(-1)^{\vert a\vert}\left(\Delta(a\cup b)-(\Delta a)\cup b-(-1)^{\vert
  a\vert}a\cup(\Delta b)\right)
\end{equation}
for $a$ and $b\in A$.
\end{defin}
\begin{rem}
In~(\ref{relation BV}), a sign
(here the sign chosen is $(-1)^{\vert a\vert}$) is needed
(See~\cite[(1.6)]{Koszul:Crochet} or~\cite[beginning of the proof of Proposition 1.2]{Getzler:BVAlg}),
since the Lie bracket of degre $+1$ is graded antisymmetric (equation~(\ref{antisymmetrie}))while
the associative product is graded commutative. Therefore the definition of Batalin-Vilkovisky algebra
in~\cite[Theorem 3.4.3 (ii)]{Ginzburg:CYalgebras} and~\cite[p. 1]{Lambre:dualiteVdB} has a sign problem.
\end{rem}
The following characterization of Batalin-Vilkovisky algebras was proved by Koszul and
rediscovered by Getzler and by Penkava and Schwarz.
\begin{proposition}~\cite[p. 3]{Koszul:Crochet}~\cite[Proposition
  1.2]{Getzler:BVAlg}~\cite{Penkava-Schwarz:algebraic}\label{equivalence
  BV 2nd order derivation}
Let $A$ be  a commutative graded algebra
$A$ equipped with an operator $\BV:A_i\rightarrow A_{i+1}$ of degree $1$
such that $\BV\circ \BV=0$.
Consider the bracket $\{\;,\;\}$ of degree $+1$ defined
by 
$$\{a,b\}=(-1)^{\vert a\vert}\left(\BV(ab)-(\BV a)b-(-1)^{\vert
  a\vert}a(\BV b)\right)$$
for any $a$, $b\in A$.
Then $A$ is a Batalin-Vilkovisky algebra
if and only if $\Delta$ is a differential operator of degree $\leq 2$,
this means that for $a$, $b$ and $c\in A$,
\begin{multline}\label{2nd order derivation}
\BV(abc)=\BV(ab)c+(-1)^{\vert a\vert} a\BV(bc)+(-1)^{(\vert a\vert -1)\vert b\vert}b\BV(ac)\\
-(\BV a)bc-(-1)^{\vert a\vert}a(\BV b)c
-(-1)^{\vert a\vert +\vert b\vert} ab(\BV c).
\end{multline}
\end{proposition}
Remark that till now, in this section, it is not necessary that the algebras have an unit.
Now if the algebras have an unit,
we give a new characterization of Batalin-Vilkovisky algebra.
One implication in this new characterization is inspired by
Ginzburg's proof of Proposition~\ref{BV algebre homologie de Hochschild}.
As we will recall in the proof of Theorem~\ref{injection d'une BV-algebre dans End},
the converse in this characterization is due to~\cite[``the restriction of this derived bracket to $A$
is the BV-bracket'', p. 1270]{Kosmann-Schwarzbach:PoissontoGerst}.
\begin{proposition}\label{equivalence BV derived bracket}
Let $A$ be  a Gerstenhaber algebra
$A$ equipped with an operator $\BV:A\rightarrow A$ of degree $1$
such that $\BV\circ \BV=0$.
For any $a\in A$, denote by $l_a:A\rightarrow A$, the left multiplication by $a$,
explicitly $l_a(b)=ab$, $b\in A$.
Denote by $[f,g]=f\circ g-(-1)^{\vert f\vert\vert g\vert}g\circ g$ the graded commutator
of two endomorphisms $f:A\rightarrow A$ and $g:A\rightarrow A$.
Then $A$ is a Batalin-Vilkovisky algebra
if and only if for $a$, $b\in A$,
$$l_{\{a,b\}}=-[[l_a,\Delta],l_b]\quad\text{and}\quad\Delta (1)=0.$$
\end{proposition}
\begin{proof}
For $a$ and $b\in A$,
\begin{multline*}
[[l_a,\Delta],l_b]=
\left(l_a\circ\Delta-(-1)^{\vert a}\Delta\circ
l_a\right)\circ l_b\\
-(-1)^{\vert b\vert(\vert a\vert+1)} l_b
\circ\left(l_a\circ\Delta-(-1)^{\vert a}\Delta\circ l_a\right)\\
=l_a\circ\Delta\circ l_b -(-1)^{\vert a\vert}\Delta\circ l_{ab}
-(-1)^{\vert b\vert}l_{ab}\circ\Delta+(-1)^{\vert b\vert(\vert a\vert
  +1)+\vert a\vert}l_b\circ\Delta\circ l_a.
\end{multline*}
Therefore  by applying this equality of operators to $c\in A$, we have
\begin{multline}\label{expression derived bracket}
-(-1)^{\vert a\vert}[[l_a,\Delta],l_b](c)=
 -(-1)^{\vert a\vert}a\Delta(bc)+\Delta(abc)\\
+(-1)^{\vert a\vert+\vert b\vert}ab\Delta(c)
-(-1)^{\vert b\vert(\vert a\vert +1)}b\Delta(ac).
\end{multline}
Suppose that $A$ is a Batalin-Vilkovisky algebra.
By Proposition~\ref{equivalence
  BV 2nd order derivation}, using~(\ref{expression derived bracket}), we obtain that
$$
-(-1)^{\vert a\vert}[[l_a,\Delta],l_b](c)=
\Delta(ab)c-(\Delta a)bc -(-1)^{\vert a\vert}a(\Delta b)c
=(-1)^{\vert a\vert}\{a,b\}c.
$$
Therefore $-[[l_a,\Delta],l_b]=l_{\{a,b\}}$.
In the case $a=b=c=1$, equation~(\ref{2nd order derivation})
gives $\Delta(1)=3\Delta(1)-3\Delta(1)=0$.

Conversely, suppose that $\Delta(1)=0$ and that
$l_{\{a,b\}}=-[[l_a,\Delta],l_b]$.
Then using~(\ref{expression derived bracket})
$$
\{a,b\}=l_{\{a,b\}}(1)=(-1)^{\vert a\vert}
\left( -(-1)^{\vert a\vert}a\Delta(b)+\Delta(ab)
+0-(\Delta a)b\right).
$$
Therefore, by Definition~\ref{definition BV algebre}, $A$ is a Batalin-Vilkovisky algebra.
\end{proof}
\section{Batalin-Vilkovisky algebra structures on Hochschild cohomology}
Let $A$ be a differential graded algebra.
The cap product defined in Section~\ref{iso cohomology homology},
$$HH_*(A,A)\otimes HH^*(A,A)\buildrel{\cap}\over\rightarrow HH_*(A,A),
c\otimes a\mapsto c\cap a$$
is a right action.

Following Tsygan definition of a calculus, we want a left action.
Therefore, we define as in~\cite[Definition 1.2]{Lambre:dualiteVdB},
$$
\mathcal{C}^*(A,A)\otimes \mathcal{C}_*(A,A)\rightarrow
\mathcal{C}_*(A,A)$$
\begin{equation}\label{passage action gauche droite}
f\otimes c\mapsto i_f(c)=f\cdot c:=(-1)^{\vert c\vert\vert f\vert}c\cap f.
\end{equation}
Explicitly
$$i_f(m[a_1|\dots|a_n]):=\sum_{p=0}^n (-1)^{\vert m\vert\vert f\vert}
(m.f[a_1|\dots|a_p])[a_{p+1}|\dots|a_n].$$
The sign in~\cite[(18)
p. 82]{Cuntz-Skandalis-Tamarkin:cyclichomnoncomgeom}
is different.
But with our choice of signs, we recover Proposition 2.6
in~\cite[p. 82]{Cuntz-Skandalis-Tamarkin:cyclichomnoncomgeom}.
Indeed for $D$, $E\in\mathcal{C}^*(A,A)$ and $c\in\mathcal{C}_*(A,A)$,
\begin{multline*}
D\cdot(E\cdot c)=(-1)^{\vert c\vert\vert E\vert} D\cdot(c\cap E)
=(-1)^{\vert c\vert \vert E\vert+\vert D\vert \vert c\vert+\vert
  D\vert \vert E\vert}
(c\cap E)\cap D\\
=(-1)^{\vert c\vert \vert E\vert+\vert D\vert \vert c\vert+\vert
  D\vert \vert E\vert}
c\cap (E\cup D)=
(-1)^{\vert D\vert\vert E\vert}(E\cup D)\cdot c
\end{multline*}
Since the cup product on $HH^*(A,A)$ is graded commutative,
for $D$, $E\in HH^*(A,A)$ and $c\in HH_*(A,A)$, we have
\begin{equation}\label{action a gauche}
D\cdot(E\cdot c)=(D\cup E)\cdot c,
\end{equation} i. e. a left action.
Note that in~\cite{Menichi:BV_Hochschild}, the author
forgot to twist the right action by the sign $(-1)^{\vert c\vert\vert
  f\vert}$, therefore has also a sign problem!
\begin{proposition}~\cite[Theorem 3.4.3 (ii)]{Ginzburg:CYalgebras}\label{BV algebre homologie de Hochschild}
Let $c\in HH_{d}(A,A)$ such that
the morphism of left $HH^*(A,A)$-modules
$$HH^p(A,A)\buildrel{\cong}\over\rightarrow
HH_{d-p}(A,A),\;a\mapsto a\cdot c$$
is an isomorphism.
If $B(c)=0$ then the Gerstenhaber algebra
$HH^*(A,A)$ equipped with $-B$ is a Batalin-Vilkovisky algebra.
\end{proposition}
\begin{proof}
Let us rewrite the proof of Victor Ginzburg
(or more precisely the proof, we already gave in~\cite[Proposition 13 and Lemma 15]{Menichi:BV_Hochschild})
using explicitly our
Proposition~\ref{equivalence BV derived bracket}
and our choice of signs.
Denote by
$$HH^p(A,A)\otimes HH_j(A,A)\rightarrow HH_{j-p+1}(A,A)$$
$$a\otimes x\mapsto L_{a}(x)
$$
the action of the suspended graded Lie algebra $HH^*(A,A)$
on $HH_*(A,A)$.
Gelfand, Daletski and
Tsygan~\cite{Daletski-Gelfand-Tsygan:variantnoncommdiffgeom} proved that the Gerstenhaber algebra
$HH^*(A,A)$ and Connes boundary map $B$ on $HH_*(A,A)$ form a calculus
~\cite[p. 93]{Cuntz-Skandalis-Tamarkin:cyclichomnoncomgeom}.
In particular, we have the two relations
$$ L_a=[B,i_{a}]$$
and~\cite[Proposition 2.9 p. 83]{Cuntz-Skandalis-Tamarkin:cyclichomnoncomgeom}
\begin{equation}\label{action d'un crochet}
i_{\{a,b\}}=(-1)^{\vert a\vert+1}[L_a,i_b].
\end{equation}
Therefore
\begin{equation}\label{gerstenhaber bracket=derived bracket}
i_{\{a,b\}}=(-1)^{\vert a\vert+1}[[B,i_{a}],i_b]=[[i_{a},B],i_b].
\end{equation}
The operator  $\Delta$ on $HH^*(A,A)$ is defined
by $$(\Delta a)\cdot c:=-B(a\cdot c)\quad\text{ for any }\quad a\in HH^*(A,A).$$
Thus $B(c)=0$ implies $\Delta(1)=0$.
Since we have a left action (equation~(\ref{action a gauche})), $l_a(b)\cdot c=(a\cup b)\cdot c=
a\cdot(b\cdot c)=i_a(b\cdot c)$
and so equation~(\ref{gerstenhaber bracket=derived bracket}) is equivalent
to
$$
l_{\{a,b\}}=-[[l_{a},\Delta],l_b].
$$
Therefore, by Proposition~\ref{equivalence BV derived bracket},
$HH^*(A,A)$ is a Batalin-Vilkovisky algebra.
\end{proof}
\begin{rem}(Signs)

i) In~\cite[Example 4.6 p. 93]{Cuntz-Skandalis-Tamarkin:cyclichomnoncomgeom}, Tsygan writes
that it follows from~\cite[2.9 p. 83]{Cuntz-Skandalis-Tamarkin:cyclichomnoncomgeom},
that $i_{\{a,b\}}=[L_a,i_b]$. We do not understand why he has no sign in this formula.
We believe that from~\cite[2.9 p. 83]{Cuntz-Skandalis-Tamarkin:cyclichomnoncomgeom}, the correct
equation with the signs is equation~(\ref{action d'un crochet}) above.

ii) In a calculus, there is a third relation, that we do not use in this paper:
$$
L_{ab}=L_{a}i_b+(-1)^{\vert a\vert}i_aL_b.
$$
Since $ab=(-1)^{\vert a\vert\vert b\vert}ba$,
$$
L_{ab}=(-1)^{\vert a\vert\vert b\vert}L_{ba}=(-1)^{\vert a\vert\vert b\vert}L_bi_a
+(-1)^{(\vert a\vert+1)\vert b\vert}i_bL_a
$$
and therefore
\begin{equation}\label{action de Lie d'un produit correct}
[L_a,i_b]=(-1)^{\vert a\vert\vert b\vert}[L_b,i_a]
\end{equation}
Since $\{a,b\}=-(-1)^{(\vert a\vert+1)(\vert b\vert+1)}\{b,a\}$,

-if we suppose like Tsygan that
 $i_{\{a,b\}}=[L_a,i_b]$, we obtain that
\begin{equation}\label{action de Lie d'un produit d'apres tsygan}
[L_a,i_b]=-(-1)^{(\vert a\vert+1)(\vert b\vert+1)}[L_b,i_a].
\end{equation}
The two equations~(\ref{action de Lie d'un produit correct}) and~(\ref{action de Lie d'un produit d'apres tsygan}) seem incoherent. Therefore we believe that the definition of calculus of Tsygan has some sign
problem.

-on the contrary, if we suppose~(\ref{action d'un crochet}),
we obtain again~(\ref{action de Lie d'un produit correct}).
\end{rem}
\section{Proof of the main theorem for path-connected groups}
\noindent{\bf Cap products associated to coalgebras}.
Let $C$ be a (differential graded) coalgebra.
Then its dual $C^\vee$ is a (differential graded) algebra.
Let $N$ be a left $C$-comodule.
Denote by $\Delta_N:N\rightarrow C\otimes N$ the structure map.
Let $\cap:N\otimes C^\vee\rightarrow N$ be the composite
\begin{equation}\label{definition cap product coalgebre}
N\otimes C^\vee\buildrel{\Delta_N\otimes C^\vee}\over\rightarrow
C\otimes N\otimes C^\vee\buildrel{C\otimes\tau}\over\rightarrow
C\otimes C^\vee\otimes N\buildrel{ev\otimes N}\over\rightarrow {\Bbbk}\otimes N\cong N.
\end{equation}
Here $\tau$ denotes the twist map given by
$n\otimes\varphi\mapsto (-1)^{\vert n\vert\vert\varphi\vert}\varphi\otimes n$
and $ev$ is the evaluation map defined by
$ev(c\otimes \varphi)=(-1)^{\vert \varphi\vert\vert c\vert}\varphi(c)$.
Then $N$ equipped with the cap product is a right
$C^\vee$-module~\cite[Proposition 2.1.1]{Sweedler:livre}.
In this paper, we are only interested in the case $N=C$.
\begin{ex}
Let $X$ be any topological space.
The (normalized or unnormalized) singular chains of $X$, $S_*(X)$ forms
a differential graded coalgebra~\cite[p. 244-5]{MacLane:hom}.
The cap product defined by~(\ref{definition cap product coalgebre}) associated to $C=S_*(X)$,
$\cap:S_*(X)\otimes S^*(X)\rightarrow S_*(X)$ is the usual cap product.
\end{ex}
\begin{ex}\label{comparaison cap produit bar construction}
Let $A$ be any augmented differential graded algebra.
Then the reduced (normalized or not) Bar construction $B(A)=\mathcal{C}_*(A,{\Bbbk})$
is a differential graded coalgebra.
The diagonal $\Delta:B(A)\rightarrow B(A)\otimes B(A)$ is given
by $$\Delta([a_1|\dots|a_n])=\sum_{p=0}^n[a_1|\dots|a_p]\otimes [a_{p+1}|\dots|a_n].$$
The cap product defined by~(\ref{definition cap product coalgebre}) associated to $C=B(A)$ is given by
$$\cap:B(A)\otimes B(A)^\vee\rightarrow B(A)$$
$$[a_1|\dots|a_n]\cap f=\sum_{p=0}^n
(-1)^{\vert f\vert(\vert a_1\vert+\dots+\vert a_n\vert+n)}f([a_1|\dots|a_p])[a_{p+1}|\dots|a_n].$$
Therefore this cap product coincides with the cap product on the Hochschild (co)chain complex
$\cap:\mathcal{C}_*(A,{\Bbbk})\otimes \mathcal{C}^*(A,{\Bbbk})
\rightarrow\mathcal{C}_*(A,{\Bbbk})$
defined by~(\ref{cap Hochschild morphism algebres}) in the case $N=B={\Bbbk}$.
\end{ex}
\begin{proposition}\label{invariance homotopique algebre a dualite de Poincare}
Let $f:C\buildrel{\simeq}\over\rightarrow D$ be a quasi-isomorphism of coalgebras. Suppose that $C$ and $D$ are ${\Bbbk}$-free.
Let $\tilde{c}\in C$ and $\tilde{d}\in D$ such that
$\tilde{d}=H_*(f)([\tilde{c}])$.
Consider the cap products defined by~(\ref{definition cap product coalgebre}) associated to the coalgebras $C$ and $D$.
Then 

the morphism of right $C^\vee$-modules $\tilde{c}\cap -:C^\vee\rightarrow C$
given by $a\mapsto \tilde{c}\cap a$ is a quasi-isomorphism
if and only if

the morphism of right $D^\vee$-modules $\tilde{d}\cap -:D^\vee\rightarrow D$
given by $a\mapsto \tilde{d}\cap a$ is a quasi-isomorphism.
\end{proposition}
\begin{proof}
The transpose of $f$: $f^\vee:D^\vee\rightarrow C^\vee$ is a morphism
of differential graded algebras.
Therefore $f^\vee$ is a morphism of right $D^\vee$-modules.
Dually, since $f$ is a morphism of coalgebras, $f$ is a morphism
of left $D$-comodules and therefore $f$
is also a morphism of right $D^\vee$-modules by~(\ref{definition cap product coalgebre}),
i. e. $f(c\cap f^\vee(\varphi))=f(c)\cap\varphi$ for any $c\in C$
and $\varphi\in D^\vee$. Note that if $f$ is the coalgebra
map $S_*(\lambda):S_*(X)\rightarrow S_*(Y)$ induced by
a continuous map $\lambda:X\rightarrow Y$, this formula is well known
(\cite[Chapter VI 5. Theorem (4)]{Bredon:topologygeometry} or~\cite[p. 241]{Hatcher:algtop}).

The composite of the morphisms of right $D^\vee$-modules
$$D^\vee\buildrel{f^\vee}\over\rightarrow C^\vee\buildrel{\tilde{c}\cap -}\over\rightarrow
C\buildrel{f}\over\rightarrow D
$$
maps $1$ to $f(\tilde{c})$ and therefore coincides with the morphism
of right $D^\vee$-modules $D^\vee\rightarrow D$, $a\mapsto f(\tilde{c})\cap a$.
Since $[\tilde{d}]=[f(\tilde{c})]$, the two maps
$a\mapsto  f(\tilde{c})\cap a$ and $a\mapsto  \tilde{d}\cap a$
coincide after passing to homology.
Therefore after passing to homology, the following square commutes
\begin{equation}\label{carre naturalite cap}
\xymatrix{D^\vee\ar[r]^{f^\vee}\ar[d]_{\tilde{d}\cap -}
& C^\vee\ar[d]^{\tilde{c}\cap -}\\
D
& C\ar[l]_{f}^\simeq}
\end{equation}
Since both $C$ and $D$ are ${\Bbbk}$-free and ${\Bbbk}$ is a principal ideal
domain, by naturality of the universal coefficient theorem for
cohomology, $H_*(f^\vee)$ is an isomorphism since
$H_*(f)$ is an isomorphism.
The proposition follows nows from the square~(\ref{carre naturalite cap}).
\end{proof}
\begin{theor}\label{cas connexe}
Let $M$ be a simply-connected oriented Poincar\'e duality space of formal dimension $d$.
Let $G$ be a topological group such that $M$ is a classifying space for $G$
or let $G$ be $\Omega M$ the (Moore) pointed loop space on $M$.
Let $[M]\in H_d(M)$ be its fundamental class.
Let $c$ the image of $[M]$ through the composite
$$
H_*(M)\buildrel{H_*(s)}\over\rightarrow H_*(LM)\buildrel{BFG^{-1}}\over\rightarrow
HH_*(S_*(G),S_*(G)).
$$
Then

a) The morphism of left $HH^*(S_*(G),S_*(G))$-modules
$$\mathbb{D}^{-1}:HH^p(S_*(G),S_*(G))\buildrel{\cong}\over\rightarrow HH_{d-p}(S_*(G),S_*(G)),
\quad a\mapsto a.c,$$
is an isomorphism.

b) The Gerstenhaber algebra $HH^*(S_*(G),S_*(G))$ equipped with the operator
$\Delta:=-\mathbb{D}\circ B\circ\mathbb{D}^{-1}$ is a Batalin-Vilkovisky algebra.
\end{theor}
Here $s$ denotes $s:M\hookrightarrow LM$ the
inclusion of the constant loops into $LM$
and $BFG$ is the isomorphism of graded ${\Bbbk}$-modules between the free loop space
homology of $M$ and the Hochschild homology of $S_*(G)$ introduced by
Burghelea, Fiedorowicz~\cite{Burghelea-Fiedorowicz:chak} and Goodwillie~\cite{Goodwillie:cychdfl}.
Finally $B$ denotes Connes boundary on
$HH_{*}(S_*(G),S_*(G))$.
\begin{rem}
We expect that the above theorem can be extended to any path-connected topological
monoid $G$ instead of just the topological monoid of pointed Moore loop spaces $\Omega M$
or instead of just any topological group.
\end{rem}
\begin{proof}
By~\cite[Proposition 6.13 in the case F=pt]{Felix-Halperin-Thomas:dgait} when $G$ is a topological group
or by~\cite[Theorem 6.3]{Felix-Halperin-Thomas:dgait} when $G=\Omega M$,
there exists a differential graded coalgebra $B(S_*(EG);S_*(G);{\Bbbk})$
and two quasi-isomorphisms of coalgebras
$$
B(S_*(G))\buildrel{\simeq}\over\leftarrow B(S_*(EG);S_*(G);{\Bbbk})\buildrel{\simeq}\over\rightarrow S_*(M).
$$
The induced isomorphism in homology is the well known isomorphism due to
Moore~\cite[Corollary 7.29]{McCleary:usersguide}
$$
\theta:\text{Tor}^{S_*(G)}({\Bbbk},{\Bbbk})= H_*(B(S_*(G)))\buildrel{\cong}\over\rightarrow H_*(M).
$$
Let $[m]\in H_*(B(S_*(G)))$ such that $\theta( [m])=[M]$.
By Proposition~\ref{invariance homotopique algebre a dualite de Poincare} and Example~\ref{comparaison cap produit bar construction},
the cap product with $[m]$, $[m]\cap -:B(S_*(G))^\vee\buildrel{\simeq}\over\rightarrow B(S_*(G))$,
$a\mapsto [m]\cap a$ is quasi-isomorphism.

Denote by $ev:LM\twoheadrightarrow M$, $l\mapsto l(0)$ the evaluation map.
The isomorphism $BFG$ of Goodwillie, Burghelea and Fiedorowicz fits into the commutative square.
$$
\xymatrix{
HH_*(S_*(G),S_*(G))\ar[r]^-{BFG}_-\cong\ar[d]_{HH_*(S_*(G),\varepsilon)}
&H_*(LM)\ar[d]_{H_*(ev)}\\
HH_*(S_*(G),{\Bbbk})\ar[r]^-{\theta}_-\cong
&H_*(M)
}
$$
Here $\varepsilon$ denote the augmentation of $S_*(G)$.
Let $c:=BFG^{-1}\circ H_d(s)([M])$.
Since $\sectiontriviale$ is a section of the evaluation map $ev$,
$HH_*(S_*(G),\varepsilon)(c)=[m]$.
So the hypotheses of statement~\ref{iso donne par cap}
are satisfied for $A=S_*(G)$.

Let $N$ be any non-negatively graded $S_*(G)$-bimodule.
Since $M$ is simply connected, by Corollary~\ref{cas simplement connexe},
we obtain that the morphism
$$
\mathcal{D}^{-1}:HH^p(S_*(G),N)\buildrel{\cong}\over\rightarrow
HH_{d-p}(S_*(G),N),\quad a\mapsto c\cap a
$$
is an isomorphism.
By taking $N=S_*(G)$ and by passing from a right action
to a left action by~(\ref{passage action gauche droite}),
we obtain a).

The isomorphism $BFG$ of Goodwillie, Burghelea and Fiedorowicz
satisfies $\Delta\circ BFG=BFG\circ B$.
Consider $M$ equipped with the trivial $S^1$-action.
The section 
$s:M\hookrightarrow LM$ is $S^1$-equivariant.
Since $$B(c)=B\circ BFG^{-1}\circ H_d(s)([M])=BFG^{-1}\circ \Delta\circ
H_d(s)([M])=0,$$
by Proposition~\ref{BV algebre homologie de Hochschild},
we obtain b).
\end{proof}
\section{Proof of the main theorem for discrete groups}
\begin{theor}\label{cas discret}
Let $G$ be a discrete group
such that its classifying space $K(G,1)$ is an oriented Poincar\'e duality space
of formal dimension $d$.
Let $[M]\in H_d(G,{\Bbbk})$ be a fundamental class.
Let $c$ be the image of $[M]$ by
$Tor^E_*(\eta,\eta):H_*(G,{\Bbbk})\rightarrow HH_*({\Bbbk}[G],{\Bbbk}[G])$
(Property~\ref{proprietes section} ii)).
Then 

a) The morphism of left $HH^*({\Bbbk}[G],{\Bbbk}[G])$-modules
$$\mathbb{D}^{-1}:HH^p({\Bbbk}[G],{\Bbbk}[G])
\buildrel{\cong}\over\rightarrow HH_{d-p}({\Bbbk}[G],{\Bbbk}[G]), a\mapsto a.c
$$
is an isomorphism.

b) The Gerstenhaber algebra $HH^*({\Bbbk}[G],{\Bbbk}[G])$
equipped with the operator $\Delta:=-\mathbb{D}\circ B\circ \mathbb{D}^{-1}$
is a Batalin-Vilkovisky algebra.
\end{theor}
\begin{proof}
Let $N$ be any ungraded ${\Bbbk}[G]$-bimodule.
Since, by hypothesis, $G$ is orientable Poincar\'e duality group,
the cap product with $[M]$ in group (co)homology gives an isomorphism
(~\cite[10.1 iv), Remark 1 and Example 1 p. 222]{Brown:cohgro}, \cite[Th 15.3.1]{Geoghegan})
$$
[M]\cap -:H^p(G,\tilde{N})\buildrel{\cong}\over\rightarrow H_{d-p}(G,\tilde{N}), a\mapsto [M]\cap a.
$$
Therefore, by Corollary~\ref{cap group egal cap Hochschild avec section}, the cap product with $c=\sigma([M])$ in Hochschild (co)homology gives
the isomorphism
$$
c\cap -:HH^p({\Bbbk}[G],N)\rightarrow HH_{d-p}({\Bbbk}[G],N), a\mapsto c\cap a.
$$
Taking $N={\Bbbk}[G]$ and passing from a right action to left action as
in~(\ref{passage action gauche droite}), we obtain a).

By i) of Property~\ref{proprietes section}, $\sigma:H_*(G;{\Bbbk})\rightarrow HH_*({\Bbbk}[G],{\Bbbk}[G])$
commute with Connes boundary map $B$ on $H_*(G;{\Bbbk})$ and on $HH_*({\Bbbk}[G],{\Bbbk}[G])$.
By a well known result of Karoubi (for example~\cite[E.7.4.8]{LodayJ.:cych}
or~\cite[Theorem 9.7.1]{Weibel:inthomalg}), Connes boundary map $B$ is trivial on $H_*(G;{\Bbbk})$.
Therefore $B(c)=B\circ \sigma([M])=\sigma\circ B([M])=0$.
By applying Proposition~\ref{BV algebre homologie de Hochschild}, we obtain b).
\end{proof}
\begin{propriete}\label{cap produit est B-lineaire}
Let $A$ and $B$ be two algebras (differential graded if we want).
Let $N$ be an $(A,A\otimes B)$-bimodule.
Let $c\in HH_d(A,A)$.
Then 

i) $HH^*(A,N)$ and $HH_*(A,N)$ are two right $B$-modules and

ii) the cap product
$$
c\cap -:HH^p(A,N)\rightarrow HH_{d-p}(A,N),\quad a\mapsto c\cap a
$$
is a morphism of right $B$-modules.
\end{propriete}
\begin{proof}
Since $N$ is an $(A^e,B)$-bimodule,
$\mathcal{C}^*(A,N)\cong\text{Hom}_{A^e}(B(A;A;A),N)$ is a (differential
graded) right $B$-module and its homology $HH^*(A,N)$ is a right $B$-module.
Similarly $\mathcal{C}_*(A,N)\cong N\otimes_{A^e}B(A;A;A)$ and $HH_*(A,N)$
are two right $B$-modules.
Let $c$ be $a[a_1|\dots|a_n]\in \mathcal{C}_n(A,A)$.
Let $f\in \mathcal{C}^p(A,N)$.
By definition,
$
c\cap f:=\pm
a f([a_1|\dots|a_p])[a_{p+1}|\dots|a_n].
$
Therefore for any $b\in B$,
\begin{multline*}
(c\cap f)\cdot b=\pm
a f([a_1|\dots|a_p])b[a_{p+1}|\dots|a_n]=\\
\pm a (f\cdot b)([a_1|\dots|a_p])[a_{p+1}|\dots|a_n]=c\cap (f\cdot b).
\end{multline*}
\end{proof}
\begin{rem}\label{bimodule qui nous interessent}
We will be only interested in the case $N=A\otimes A$ and $B=A^e$.
Here the $A$-bimodule structure on $N$ is given by
$a\cdot (x\otimes y)\cdot b=ax\otimes yb$ and is called the {\it outer}
structure~\cite[(1.5.1)]{Ginzburg:CYalgebras}.
And the right $B$-module on $N$ is given by
$(x\otimes y)\cdot (a\otimes b)=xa\otimes by$, $x\otimes y\in N$, $a\otimes b\in B$ and is called the {\it inner} structure.
\end{rem}
\begin{defin}(\cite[Definition 3.2.3, (3.2.5), Remark 3.2.8]{Ginzburg:CYalgebras} or simply~\cite[Definition 2.1]{Berger-Taillefer})\label{definition algebre de Calabi-Yau}
An ungraded algebra $A$ is {\it Calabi-Yau} of dimension $d$ if

i) viewed as an $A$-bimodule over itself, $A$ admits a finite resolution
by finite type projective $A$-bimodules, i. e. there exists
an exact sequence of $A^e$-projective finite type module of the form
$$
0\rightarrow P_i\rightarrow P_{i-1}\rightarrow\dots\rightarrow P_1\rightarrow
P_0\rightarrow A\rightarrow 0,
$$

ii) for all $k\neq d$, $HH^k(A,A\otimes A)=0$ and

iii) as $(A,A)$-bimodule, $HH^d(A,A\otimes A)$ is isomorphic to $A$
(Here the $(A,A)$-bimodule on $HH^*(A,A\otimes A)$ is given by
Property~\ref{cap produit est B-lineaire}
and Remark~\ref{bimodule qui nous interessent}).
\end{defin}
\begin{proposition}(Stated without proof in~\cite[Remark 3.4.2]{Ginzburg:CYalgebras})\label{dualite de Van den Bergh implique Calabi-Yau}
Let $A$ be an ungraded algebra. Let $c\in HH_d(A,A)$.
Suppose that for every $A$-bimodule $N$,
$c\cap -:HH^p(A,N)\buildrel{\cong}\over\rightarrow HH_{d-p}(A,N)$, $a\mapsto c\cap a$, is an isomorphism. Then $A$ satisfies conditions ii) and iii) of Definition~\ref{definition algebre de Calabi-Yau}.
\end{proposition}
\begin{proof}
Let $N$ be a free $(A,A)$-bimodule. Then $HH_*(A,N)=0$ if $*\neq 0$.
Therefore $HH^k(A,N)=0$ if $k\neq d$.
Suppose moreover that $N$ is a $(A,A\otimes B)$-bimodule.
The quasi-isomorphism of complexes
$\mathcal{C}_*(A,N)\cong N\otimes_{A^e}B(A;A;A)
\buildrel{\simeq}\over\rightarrow N\otimes_{A^e}A$
is a morphism of right $B$-modules.
By Property~\ref{cap produit est B-lineaire},
$$
c\cap -:HH^d(A,N)\rightarrow HH_{0}(A,N)\cong N\otimes_{A^e} A
$$
is an isomorphism of right $B$-modules.

Let $N$ be the $(A,A)$-bimodule $A\otimes A$ with the outer structure
and $B=A^e$ (See Remark~\ref{bimodule qui nous interessent}).
Then $N\otimes_{A^e} A=(A\otimes A)\otimes_{A^e} A
\buildrel{\cong}\over\rightarrow A$, $(x\otimes y)\otimes_{A^e} m\mapsto ymx$
is an isomorphism whose inverse is the map mapping $a\mapsto (1\otimes 1)\otimes_{A^e} a$. A straightforward calculation shows that theses isomorphisms
are right $A^e$-linear. Therefore, we have proved that
$HH^d(A,A\otimes A)$ is isomorphic to $A$ as right $A^e$-modules.
\end{proof}
\begin{theor}
Let ${\Bbbk}$ be any commutative ring.
Let $G$ be a orientable Poincar\'e duality group of dimension $d$.
Then its group ring
${\Bbbk}[G]$ is a Calabi-Yau algebra of dimension $d$.
\end{theor}
When ${\Bbbk}$ is a field of characteristic $0$
or of characteristic prime to the cardinal of $G$,
this theorem was proved
by Kontsevich~\cite[Corollary 6.1.4]{Ginzburg:CYalgebras}
and Lambre~\cite[Lemme 6.2]{Lambre:dualiteVdB}.
\begin{proof}
By~\cite[Remark 2. p. 222]{Brown:cohgro}, there exists
a finite resolution $P\buildrel{\simeq}\over\rightarrow {\Bbbk}$
of ${\Bbbk}$ by finite type projective ${\Bbbk}[G]$-left modules.
Then $X:={\Bbbk}[G\times G^{op}]\otimes_{{\Bbbk}[G]}P\buildrel{\simeq}\over\rightarrow {\Bbbk}[G]$ is a finite type resolution of  ${\Bbbk}[G]$
by finite type projective ${\Bbbk}[G]$-bimodules.

In the proof of Theorem~\ref{cas discret},we saw that for any
${\Bbbk}[G]$-bimodule $N$,
$c\cap -:HH^p({\Bbbk}[G],N)\buildrel{\cong}\over\rightarrow HH_{d-p}({\Bbbk}[G],N)$, $a\mapsto c\cap a$, is an isomorphism.
Therefore, by Proposition~\ref{dualite de Van den Bergh implique Calabi-Yau},
${\Bbbk}[G]$ is a Calabi-Yau algebra
of dimension $d$.
\end{proof}
\section{String topology of classifying spaces}
In~\cite{Chataur-Menichi:stringclass}, Chataur and the author,
and in~\cite{BGNX:Stringstacks},
Behrend, Ginot, Noohi and Xu developped a string topology theory
dual to Chas-Sullivan string topology.
\begin{theor}~\cite{BGNX:Stringstacks,Chataur-Menichi:stringclass}\label{BV structure sur cohomologie des lacets libres du classifiant}
Let $G$ be a path-connected compact Lie group of dimension $d$.
Denote by $BG$ its classifying space.
Then the shifted free loop space cohomology $H^{*+d}(LBG)$
is a (possibly non-unital) Batalin-Vilkovisky algebra.
\end{theor}
The goal of this section is to prove the following theorem:
\begin{theor}\label{BV structure sur cohomologie de Hochschild du classifiant}
Let $G$ be a path-connected compact Lie group of dimension $d$.
Denote by $S^*(BG)$ the singular cochains on the classifying space of $G$.
Then

a) There exists an explicit isomorphism of left $HH^*(S^*(BG),S^*(BG))$-modules
$$\mathbb{D}^{-1}:HH^p(S^*(BG),S^*(BG))\buildrel{\cong}\over\rightarrow
HH_{-d-p}(S^*(BG),S^*(BG)).$$

b) The Gerstenhaber algebra $HH^*(S^*(BG),S^*(BG))$ equipped with the operator
$\Delta:=-\mathbb{D}\circ B\circ\mathbb{D}^{-1}$ is a Batalin-Vilkovisky algebra.
\end{theor}
Both Batalin-Vilkovisky algebras in Theorems~\ref{BV structure sur cohomologie des lacets libres du classifiant} and~\ref{BV structure sur cohomologie de Hochschild du classifiant} are determined by an orientation class of $H_d(G)$.
In~\cite{JonesJ:Cycheh}, Jones gave
an isomorphism of graded vector spaces
$$
J:HH_*(S^*(BG),S^*(BG))\buildrel{\cong}\over\rightarrow H^*(LBG).
$$
Again, we conjecture that the isomorphism of graded vector spaces
$J\circ \mathbb{D}^{-1}:HH^*(S^*(BG),S^*(BG))\buildrel{\cong}\over\rightarrow
H^{*+d}(LBG)$ is a morphism of Batalin-Vilkovisky algebras.

Theorem~\ref{BV structure sur cohomologie de Hochschild du classifiant}
is the Eckmann-Hilton or Koszul dual of the following theorem
proved by Chataur and the author.
\begin{theor}\cite[Theorem 54]{Chataur-Menichi:stringclass}
Let $G$ be a path-connected compact Lie group of dimension $d$.
Denote by $S_*(G)$ the algebra of singular chains of $G$.
Consider Connes coboundary map $H(B^\vee)$
on the Hochschild cohomology of $S_*(G)$ with coefficients in its dual,
$HH^{*}(S_*(G);S^*(G))$.
then there is an isomorphism of graded vector spaces of upper degree $d$
$$\mathcal{D}^{-1}:HH^p(S_*(G);S_*(G))\buildrel{\cong}\over\rightarrow
HH^{p+d}(S_*(G);S^*(G))$$
such that the Gerstenhaber algebra $HH^*(S_*(G);S_*(G))$
equipped with the operator
$\Delta=\mathcal{D}\circ H(B^\vee)\circ \mathcal{D}^{-1}$ is a Batalin-Vilkovisky
algebra. 
\end{theor}
\subsection{Frobenius algebras}.
\begin{defin}\label{definition algebre de Frobenius}
Let $A$ be a differential graded algebra.
We say that $A$ is a {\it Frobenius algebra}
if there is a quasi-isomorphism of right $A$-modules
$A\buildrel{\simeq}\over\rightarrow A^\vee$.
In particular, a graded algebra $A$ is a Frobenius algebra
if $A$ is isomorphic as right $A$-modules to its dual 
$A^\vee$.
\end{defin}
\begin{propriete}~\cite[Theorem 9.8]{Lu-Palmieri-Wu-Zhang}\label{algebre de Frobenius est une propriete homologique}
Let $A$ be a differential graded algebra.
Then $A$ is a Frobenius algebra if and only if its homology
$H(A)$ is a Frobenius algebra.
\end{propriete}
\begin{proof}
Let $M$ be any left $A$-module.
A straightforward computation shows that the linear map
$\mu:H(\text{Hom}(M,{\Bbbk}))\rightarrow
\text{Hom}(H(M),{\Bbbk})$ mapping a cycle $f:M\rightarrow{\Bbbk}$
to $H(f):H(M)\rightarrow {\Bbbk}$ is a morphism of right
$H(A)$-modules.
Since in this section, ${\Bbbk}$ is a field, by the universal coefficient
theorem for cohomology, this map $\mu$ is an isomorphism.
We are only interested in the case $M=A$.

Suppose that we have an quasi-isomorphism of right $A$-modules
$\Theta:A\buildrel{\simeq}\over\rightarrow A^\vee$. Then the composite
$H(A)\buildrel{H(\Theta)}\over\rightarrow H(A^\vee)
\buildrel{\mu}\over\rightarrow H(A)^\vee$ is an isomorphism
of right $H(A)$-modules.

Conversely, suppose that we have an isomorphism of right
$H(A)$-modules,
$\Theta:H(A)\buildrel{\cong}\over\rightarrow H(A)^\vee$.
Then the composite
$H(A)\buildrel{\Theta}\over\rightarrow H(A)^\vee
\buildrel{\mu^{-1}}\over\rightarrow H(A^\vee)$ is also an isomorphism
of right $H(A)$-modules.
Let $x$ be a cycle of $A^\vee$ such that
$\mu^{-1}\circ\Theta(1)=[x]$.The morphism of right $A$-modules
$A\rightarrow A^\vee$, $a\mapsto xa$, coincides in homology
with the isomorphism $\mu^{-1}\circ\Theta$.
\end{proof}
\begin{cor}
Let $A$ and $B$ be two differential graded algebras
such that $H(A)\cong H(B)$ as graded algebras.
Then $A$ is Frobenius if and only if $B$ is.
\end{cor}
Note that it is not necessary that there is a quasi-isomorphism
of algebras $f:A\buildrel{\simeq}\over\rightarrow B$
(Compare with Proposition~\ref{invariance homotopique algebre a dualite de Poincare} or~\cite[Corollary 9.9]{Lu-Palmieri-Wu-Zhang}).
\begin{propriete}\label{forme bilineaire algebre coalgebre}
Let $A$ be a graded algebra and let $C$ be a graded coalgebra.
Consider a bilinear form $<\;,\;>:C\otimes A:\rightarrow {\Bbbk}$.
Let $\phi:A\rightarrow C^\vee$, $a\mapsto <-,a>$
and let $\psi:C\rightarrow A^\vee$, $c\mapsto <c,->$ be the right
and left adjoints.
Suppose that $\phi$ is a morphism of graded algebras.
Then 

i) $\psi$ is a morphism of right $A$-modules
with respect to the cap
product~(\ref{definition cap product coalgebre})
associated to the coalgebra $C$
, i. e.
$\psi(c\cap\phi(a))=\psi(c).a$ for any $c\in C$ and $a\in A$.

ii) If $A$ is non-negatively graded and of finite type in each degre
then $\psi:C\rightarrow A^\vee$ is a morphism of graded coalgebras.
\end{propriete}
\begin{proof}
Let $\Delta c=\sum c'\otimes c"$ be the diagonal of $c$.
By definition, the cap product $c\cap\phi(a)$ is
equal to $\sum\pm <c',a>c"$.
Therefore $\psi(c\cap\phi(a))$ is the form on $A$, mapping
$x\in A$ to $\sum\pm <c',a><c",x>$.
On the other hand, $\psi(c).a$ is the form on $A$ mapping
$x\in A$ to $<c,ax>$. But $\psi$ is a morphism of algebras if only
and if for every $a$, $x\in A$ and $c\in C$,
$<c,ax>=\sum\pm <c',a><c",x>$.
\end{proof}
Let us give a well-known application of i) of
Property~\ref{forme bilineaire algebre coalgebre}.
Let $C=S_*(M)$ and $A=C^\vee=S^*(M)$.
We obtain that the quasi-isomorphism $\psi:S_*(M)\rightarrow
S^*(M)^\vee$ is a morphism of
$S^*(M)$-modules~\cite[Section 7]{Felix-Thomas-Vigue:Hochschildmanifold}.
Therefore by Poincar\'e duality, $S^*(M)$ is a Frobenius algebra.
And $H^*(M)$ also.
\subsection{String topology of manifolds}.
Let $M$ be a closed oriented $d$-dimensional smooth manifold.
Denote by $\mathbb{H}_*(M):=H_{*+d}(M)$.
Poincar\'e duality~\cite[Theorem 3.30]{Hatcher:algtop}
gives an isomorphism of graded algebras
$$H^*(M)\cong\mathbb{H}_{*}(M)$$ where

-the product on $H^*(M)$ is the cup product $H^*(\Delta)$,

-the product on $\mathbb{H}_{*}(M)$ is the intersection product $\Delta_!$
and

-the fundamental class $[M]\in H_d(M)$ is the unit of $\mathbb{H}_{*}(M)$.
 
Chas and Sullivan have defined a Batalin-Vilkovisky algebra on
$\mathbb{H}_*(LM):=H_{*+d}(LM)$.
The Chas-Sullivan loop product on $\mathbb{H}_*(LM)$
mixes the intersection product $\Delta_!$ on $\mathbb{H}_{*}(M)$
and the Pontryagin product $H_*(comp)$ on $H_*(\Omega M)$.

More precisely, let $\tilde{\Delta}:M^{S^1\vee S^1}\hookrightarrow LM\times LM$
be the inclusion map and let $comp:M^{S^1\vee S^1}\rightarrow LM$
be the map obtained by composing loops.
The Chas-Sullivan loop product is the composite
$$
H_*(LM\times LM)\buildrel{\tilde{\Delta}_!}\over\rightarrow
H_{*-d}(M^{S^1\vee S^1})\buildrel{H_*(comp)}\over\rightarrow
H_{*-d}(LM).
$$
The loop product admits $H_d(s)([M])$ as unit.
More generally $H_*(s):\mathbb{H}_*(M)\rightarrow\mathbb{H}_*(LM)$
is a morphism of algebras preserving the units.
Let $i:\Omega M\hookrightarrow LM$ be the inclusion of the pointed loops
into the free loops.
The shriek map of $i$, called the intersection map,
$i_!:\mathbb{H}_*(LM)\rightarrow H_*(\Omega M)$,
is also a morphism of algebras preserving the units~\cite[Proposition 3.4]{Chas-Sullivan:stringtop}.

The unit of the Batalin-Vilkovisky algebra $\mathbb{H}_*(LM)$
and the fact that $\Delta 1=0$ in any unital Batalin-Vilkovisky
algebras was the key for proving Theorem~\ref{cas connexe}.

\subsection{\bf Versus string topology of classifying spaces}.
Let $G$ be a path-connected Lie group of dimension $d$.
Denote by $\mathbb{H}^*(\Omega BG)=H^{*+d}(\Omega BG)$.
Since $H_*(\Omega BG)$ is a finite dimensional Hopf algebra,
$H_*(\Omega BG)$ is a Frobenius algebra:
there is an isomorphism of right $H_*(\Omega BG)$-modules~\cite[Section 4.1]{Chataur-Menichi:stringclass}
$$\Theta:H_*(\Omega BG)\cong \mathbb{H}^*(\Omega BG).$$
By~\cite[Theorem 5.1.2, with left Hopf modules instead of right Hopf modules]{Sweedler:livre}, 
the composite of the antipode of the Hopf algebra $H_*(\Omega BG)$ and of $\Theta$, 
$
H_*(\Omega BG)\buildrel{S}\over\rightarrow 
H_*(\Omega BG)\buildrel{\Theta}\over\rightarrow H^*(\Omega BG)
$
is an isomorphism of left Hopf modules over $H_*(\Omega BG)$, and so coincides with
Poincar\'e duality.

Therefore this isomorphism $\Theta$ is an isomorphism of algebras
if

-the product on $H_*(\Omega BG)$ is the Pontryagin product $H_*(comp)$,

-the product on $\mathbb{H}^*(\Omega BG)$ is the composite
$$H^*(\Omega BG)\otimes H^*(\Omega BG)
\buildrel{\tau}\over\rightarrow H^*(\Omega BG)\otimes H^*(\Omega BG)
\buildrel{comp^!}\over\rightarrow H^{*-d}(\Omega BG)
$$
where $\tau$ denote the twist map given by $a\otimes b\mapsto
(-1)^{\vert a\vert\vert b\vert}b\otimes a$ and
$comp^!$ is the shriek map of $comp$.

Of course, $\Theta(1)$ is the unit of the algebra $\mathbb{H}^*(\Omega BG)$.

The product on $\mathbb{H}^*(LBG):=H^{*+d}(LBG)$
mixes the cup product $H^*(\Delta)$ on $H^*(BG)$
and the product $comp^!$ on $\mathbb{H}^*(\Omega BG)$.
More precisely, the product on $\mathbb{H}^*(LBG)$ is the
composite
$$
H^*(LBG\times LBG)\buildrel{H^*(\tilde{\Delta})}\over\rightarrow
H^*(BG^{S^1\vee S^1})\buildrel{comp^!}\over\rightarrow H^{*-d}(LBG).
$$
Comparing with the definition of the Chas-Sullivan loop product
defined above, we see a general principle. In order to pass from string topology of manifolds to string
topology of classifying spaces, you replace

-homology by cohomology,

-shriek map in homology like $\tilde{\Delta}_!$ by the map
induced in singular cohomology like $H^*(\tilde{\Delta})$,

-maps induced in singular homology like $H_*(comp)$
by shriek map in cohomology like $comp^!$.

\noindent In particular, you never change the direction of arrows.

Guided by this general principle, we now transpose the proof
of Theorem~\ref{cas connexe} into a proof of Theorem~\ref{BV structure sur cohomologie de Hochschild du classifiant}.
Using this general principle, the product on $\mathbb{H}^*(LBG)$
should have $s^!(1)$ as an unit.
More generally $s^!:H^*(BG)\rightarrow \mathbb{H}^*(LBG)$
should be a morphism of algebras preserving the units.
Also $H^*(i):\mathbb{H}^*(LBG)\rightarrow \mathbb{H}^*(\Omega BG)$
should be a morphism of algebras preserving the units.
The problem is that $s^!$ is not easy to define
~\cite[Remark 56]{Chataur-Menichi:stringclass}
and that we have not yet proved the previous assertions.
Instead, we are going only to prove the following lemma.
\begin{lem}\label{remplacant unite de la BV-algebre}
There exists an explicit element $\mathbb{I}\in H^d(LBG)$
such that $\Delta \mathbb{I}=0$
and such that the morphism of right $H_*(\Omega BG)$-modules,
$\Theta:H_p(\Omega BG)\buildrel{\cong}\over\rightarrow H^{d-p}(\Omega BG)$, $a\mapsto H^d(i)(\mathbb{I}).a$ is an isomorphism.
\end{lem}
As explained above, we believe that $\mathbb{I}$ is the unit
of the Batalin-Vilkovisky algebra $\mathbb{H}^*(LBG)$.
\begin{proof}
Let $\eta:\{e\}\rightarrow G$ be the unit of $G$.
Consider $\eta_!:H_d(G)\rightarrow {\Bbbk}$ the shriek map of $\eta$.
By Lemma 55 of~\cite{Chataur-Menichi:stringclass}, the morphism of right
$H_*(G)$-modules $H_p(G)\buildrel{\cong}\over\rightarrow H^{d-p}(G)$,
$a\mapsto \eta_!.a$, is an isomorphism.
Consider the commutative diagram of graded algebras
$$
\xymatrix{
H^*(LBG)\ar[r]^{H^*(\gamma)}_\cong\ar[d]_{H^*(i)}
& H^*(\vert\Gamma G\vert)\ar[d]_{H^*(\vert j\vert)}
& H^*(EG\times_G G^{ad})\ar[l]_-{H^*(\vert \Phi)\vert}^-\cong
\ar[dl]^{H^*(E\eta\times_\eta G^{ad})}\\
H^*(\Omega BG)\ar[r]^{H^*(\bar{\gamma})}_\cong
& H^*(G)
}$$
where the right triangle is the triangle considered in the proof of
Theorem 54 of~\cite{Chataur-Menichi:stringclass}
and the left square is induced by the following
commutative square of topological spaces 
\xymatrix{
G\ar[r]^{\vert j\vert}\ar[d]^{\bar{\gamma}}_\simeq
& \vert\Gamma G\vert\ar[d]^{\gamma}_\simeq\\
\Omega BG\ar[r]^i
&LBG
}
considered in the proof of Theorem 7.3.11 of~\cite{LodayJ.:cych}.
Consider the equivariant Gysin map:
$EG\times_G\eta^!:H^*(BG)\rightarrow H^{*+d}(EG\times_G G^{ad})$.
Let $\mathbb{I}$ be the image of $1$ by the composite $H^*(\gamma)^{-1}\circ H^*(\vert\Phi\vert)\circ EG\times_G\eta^!$.
In~\cite[(58)]{Chataur-Menichi:stringclass}, we saw that $\Delta\mathbb{I}=0$.
By Lemma 57 of~\cite{Chataur-Menichi:stringclass}, 
$ H^*(E\eta\times_\eta G^{ad})$ maps $EG\times_G\eta^!(1)$ to $\eta_!\in H_d(G)^\vee$. Therefore using the above commutative diagram, 
$H^*(i)(\mathbb{I})=H^*(\bar{\gamma})^{-1}(\eta_!)$.

By Lemma 7.3.12 of~\cite{LodayJ.:cych}, $\bar{\gamma}:G\buildrel{\simeq}\over\rightarrow \Omega BG$ is the classical homotopy equivalence which is well-known
to be a morphism of $H$-spaces.
Therefore the isomorphism induced in homology, $H_*(\bar{\gamma}):H_*(G)\buildrel{\cong}\over\rightarrow H_*(\Omega BG)$, is a morphism of algebras.
Since $H_*(G)$ is a Frobenius algebra, $H_*(\Omega BG)$ is also a Frobenius algebra. More precisely, the morphism of right $H_*(\Omega BG)$-modules
$\Theta:H_p(\Omega BG)\rightarrow H_{d-p}(\Omega BG)^\vee$, $a\mapsto H^*(\bar{\gamma})^{-1}(\eta_!).a$ is an isomorphism.
\end{proof}
To finish the proof of Theorem~\ref{BV structure sur cohomologie de Hochschild du classifiant}, we need also the following algebraic results.
\subsection{Bar and Cobar construction}
Let $C$ be a coaugmented DGC. Denote by $\overline{C}$ the kernel of the counit.
The normalized {\it cobar construction on $C$}, denoted $\Omega C$, 
is the augmented differential graded algebra
$\left(T(s^{-1}\overline{C}),d_1+d_2\right)$
where $d_1$ and $d_2$ are the unique derivations determined by
$$d_1s^{-1}c=-s^{-1}dc\mbox{ and}$$
$$d_2s^{-1}c=\sum_i (-1)^{\vert x_i\vert} s^{-1}x_i\otimes s^{-1}y_i,\; c\in\overline{C}$$
where the reduced diagonal $\displaystyle\overline{\Delta}c=\sum_i x_i\otimes y_i$. We follow the sign convention of \cite{Felix-Halperin-Thomas:Adamce}.

\begin{rem}~\cite[(A.6)]{Halperin:unieal}\label{extension bilinear form}
A bilinear form $<\;,\;>:V\otimes W\rightarrow {\Bbbk}$ of graded vector spaces
extends a bilinear form $<\;,\;>:TV\otimes TW\rightarrow {\Bbbk}$
defined by
$$<v_1\otimes\dots\otimes v_i, w_1\otimes\dots\otimes w_i>=\pm\prod_{j=1}^i <v_j,w_j>$$
and $<v_1\otimes\dots\otimes v_i, w_1\otimes\dots\otimes w_j>=0$
if $i\neq j$. Here again $\pm$ is the sign given by the Koszul sign convention.
\end{rem}
\begin{proposition}\label{dualite bar cobar}
Let $C$ be a coaugmented differential graded coalgebra.
Denote by $A:=C^\vee$ the differential graded algebra dual of $C$.
Let $<\;,\;>:sA\otimes s^{-1}C\rightarrow {\Bbbk}$ be the non-degenerate
bilinear form defined in~\cite[p. 276 in the case $V=s^{-1}C$ and $X=A$]{Halperin:unieal}
by $<sa,s^{-1}c>=(-1)^{\vert a \vert+1}a(c)$.
Consider the bilinear form $<\;,\;>:BA\otimes \Omega C\rightarrow {\Bbbk}$
extending $<\;,\;>:sA\otimes s^{-1}C\rightarrow {\Bbbk}$ (Remark~\ref{extension bilinear form}). Then

i) the right adjoint $\phi:\Omega C\rightarrow BA^\vee$ is a natural
morphism of differential graded algebras
and the left adjoint $\psi:BA\rightarrow \Omega C^\vee$ is a natural
morphism of complexes,

ii) if $C$ is of finite type in each degre and $C={\Bbbk}\oplus C_{\geq 2}$
then both $\phi$ and $\psi$ are isomorphisms,
 
iii) if $H(C)$ is of finite type in each degre and $C={\Bbbk}\oplus C_{\geq 2}$
then both $H(\phi)$ and $H(\psi)$ are isomorphisms.
\end{proposition}
\begin{proof}
i) and ii) Denote by $TAW$ the tensor algebra on $W$, and by $TCV$ the
tensor coalgebra on $V$~\cite[p. 277-8]{Halperin:unieal}.
It is easy to check that the right adjoint map $\phi:TAW\rightarrow TCV^\vee$
of the bilinear map defined by Remark~\ref{extension bilinear form} is a morphism of graded algebras.
In~\cite[Proof of Theorem 6.1 ii)]{MenichiL:cohrfl}, we have checked carefully
that $\psi:\mathcal{C}_*(A,A)\rightarrow (C\otimes\Omega C,\delta)^\vee$,
where $(C\otimes\Omega C,\delta)$ is the cyclic cobar complex of $C$,
is a morphism of complexes and an isomorphism if $C$ is of finite type
in each degre and $\overline{C}=\overline{C}_{\geq 2}$.
The same proof shows that this is also the case for
$\psi:BA\rightarrow\Omega C^\vee$.

iii) By Proposition 4.2 of~\cite{Felix-Halperin-Thomas:Adamce},
there exists a differential graded algebra of the form $(TV,d)$
where $V=V^{\geq 2}$ is of finite type in each degre and a quasi-isomorphism of
augmented differential graded algebras
$f:TV\buildrel{\simeq}\over\rightarrow C^\vee$.
By ii) of Property~\ref{forme bilineaire algebre coalgebre},
the adjoint map $g:C\buildrel{\simeq}\over\rightarrow (C^\vee)^\vee\build\rightarrow_\simeq^{f^\vee} TV^\vee$ is a quasi-isomorphism of coaugmented
differential graded coalgebras~\cite[p. 56]{Felix-Menichi-Thomas:GerstduaiHochcoh}. Denote by $D:=TV^\vee$.

Since $\overline{C}_{\leq 1}=\overline{D}_{\leq 1}=0$, by Remark 2.3 of~\cite{Felix-Halperin-Thomas:Adamce},
$\Omega f:\Omega C\buildrel{\simeq}\over\rightarrow\Omega D$ is a quasi-isomorphism of augmented differential graded algebras.
Since ${\Bbbk}$ is a field,
$f^\vee:D^\vee\buildrel{\simeq}\over\rightarrow C^\vee$ is also a quasi-isomorphism
of augmented differential graded algebras.
By naturality of $\psi$, we have the commutative square of complexes
\xymatrix{
B(C^\vee)\ar[r]^\psi
&(\Omega C)^\vee\\
B(D^\vee)\ar[r]^\psi_\cong\ar[u]^{B(f^\vee)}_\simeq
&(\Omega D)^\vee\ar[u]_{(\Omega f)^\vee}^\simeq
}
where the two vertical morphisms are quasi-isomorphisms.
By ii), $\psi:B(D^\vee)\buildrel{\cong}\over\rightarrow(\Omega D)^\vee$
is an isomorphism. Therefore $\psi:B(C^\vee)\buildrel{\simeq}\over\rightarrow(\Omega C)^\vee$ is a quasi-isomorphism.
Similarly, one proves that $\phi:\Omega C\buildrel{\simeq}\over\rightarrow B(C^\vee)^\vee$ is also a quasi-isomorphism.
\end{proof}
\begin{proof}[Proof of Theorem~\ref{BV structure sur cohomologie de Hochschild du classifiant}]
The Eilenberg Moore formula gives an isomorphism of graded algebras
$\mathcal{EM}:H_*(\Omega BG)\buildrel{\cong}\over\rightarrow
H(\Omega S_*(BG))$.
By Proposition~\ref{dualite bar cobar} iii),
$\psi:BS^*(BG)\buildrel{\simeq}\over\rightarrow \Omega S_*(BG)^\vee$
is a quasi-isomorphism of complexes.
The Jones isomorphism $J$ fits into the commutative diagram
$$
\xymatrix{
HH_*(S^*(BG),S^*(BG))\ar[rr]^{J}_\cong\ar[d]_{HH_*(S^*(BG),\varepsilon)}
&& H^*(LBG)\ar[d]^{H^*(i)}\\
\text{Tor}^{S^*(BG)}({\Bbbk},{\Bbbk})\ar[r]^{H(\psi)}_\cong
& H(\Omega S_*(BG))^\vee\ar[r]^-{\mathcal{EM}^\vee}_-\cong
&H^*(\Omega BG)
}$$
Consider the element $\mathbb{I}\in H^d(LBG)$ given by
Lemma~\ref{remplacant unite de la BV-algebre}. Let $c$ be $J^{-1}(\mathbb{I})\in HH_{-d}(S^*(BG),S^*(BG))$.
Denote by $m\in BS^*(BG)$ a cycle such that its class $[m]$ is
equal to $HH_{-d}(S^*(BG),\varepsilon)(c)$.

Since $H_*(\Omega BG)$ is a Frobenius algebra, $H(\Omega S_*(BG))$ is
also a Frobenius algebra.
More precisely, by Lemma~\ref{remplacant unite de la BV-algebre},
the morphism of right $H_*(\Omega BG)$-modules
$H_p(\Omega BG)\buildrel{\cong}\over\rightarrow H_{d-p}(\Omega BG)^\vee$
mapping $1$ to $ H^d(i)(\mathbb{I})$ is an isomorphism.
Therefore the morphism of right $H(\Omega S_*(BG))$-modules
$H_p(\Omega S_*(BG))\buildrel{\cong}\over\rightarrow H_{d-p}(\Omega S_*(BG))^\vee$
mapping $1$ to $(\mathcal{EM}^\vee)^{-1}\circ H^d(i)(\mathbb{I})$ is an isomorphism.
Since the above diagram is commutative,
$(\mathcal{EM}^\vee)^{-1}\circ H^d(i)(\mathbb{I})=H(\psi)([m])$.
By Property~\ref{algebre de Frobenius est une propriete homologique}, the differential graded algebra $\Omega S_*(BG)$ is a Frobenius algebra.
More precisely, the morphism of right $\Omega S_*(BG)$-modules
$\theta:\Omega S_*(BG)\buildrel{\simeq}\over\rightarrow (\Omega S_*(BG))^\vee$,
$a\mapsto \psi(m).a$ is a quasi-isomorphism.

By Proposition~\ref{dualite bar cobar},
$\phi:\Omega S_*(BG)\buildrel{\simeq}\over\rightarrow BS^*(BG)^\vee$
is a quasi-isomorphism of differential graded algebras.
Therefore by i) of Property~\ref{forme bilineaire algebre coalgebre},
the following square of complexes commutes.
$$
\xymatrix{
\Omega S_*(BG)\ar[r]^\phi_\simeq\ar[d]_\theta^\simeq
&(BS^*(BG))^\vee\ar[d]^{m\cap -}_\simeq\\
(\Omega S_*(BG))^\vee
&BS^*(BG)\ar[l]^\psi_\simeq
}$$
Therefore (Example~\ref{comparaison cap produit bar construction}),
$$
[m]\cap -:\text{Ext}^p_{S^*(BG)}({\Bbbk},{\Bbbk})
\buildrel{\cong}\over\rightarrow \text{Tor}_{-d-p}^{S^*(BG)}({\Bbbk},{\Bbbk})
$$
is an isomorphism.

Let $N$ be any non-negatively upper graded $S^*(BG)$-bimodule.
Since $BG$ is path-connected, by Corollary~\ref{cas cochain connexe},
we obtain that the morphism
$$
\mathcal{D}^{-1}:HH^p(S^*(BG),N)\buildrel{\cong}\over\rightarrow
HH_{-d-p}(S^*(BG),N),\quad a\mapsto c\cap a
$$
is an isomorphism.
By taking $N=S^*(BG)$ and by passing from a right action
to a left action by~(\ref{passage action gauche droite}),
we obtain a).

The isomorphism $J$ of Jones
satisfies $\Delta\circ J=J\circ B$.
Since by Lemma~\ref{remplacant unite de la BV-algebre},
$$B(c)=B\circ J^{-1}(\mathbb{I})=J^{-1}\circ \Delta(\mathbb{I})=0,$$
by Proposition~\ref{BV algebre homologie de Hochschild},
we obtain b).

\end{proof}
\section{Appendix}
The key of the proof of Proposition~\ref{BV algebre homologie de Hochschild} is the relation
$$
i_{\{a,b\}}=(-1)^{\vert a\vert+1}[[B,i_{a}],i_b]=[[i_{a},B],i_b].
$$
In this appendix, we recall that $[[i_{a},B],i_b]$ is the {\it derived bracket} of $i_{a}$ and $i_b$
and we explain that this relation means that the morphism of graded algebras
$$HH^*(A,A)\rightarrow \text{End}(HH_*(A,A)),\quad a\mapsto i_a,$$
is a morphism of generalized Loday-Gerstenhaber algebras
(Theorem~\ref{hochschild morphism de gerstenhaber})
\begin{defin}\cite[p. 1247]{Kosmann-Schwarzbach:PoissontoGerst}\label{definition algebre de Loday_Gerstenhaber}
A {\it generalized Loday-Gerstenhaber algebra} is a (not necessarily commutative) graded algebra $A$
equipped with a linear map
$\{-,-\}:A_i \otimes A_j \to A_{i+j+1}$ of degree $1$
such that:

\noindent a) the bracket $\{-,-\}$ gives $A$ a structure of graded
Leibniz algebra of degree $1$. This means that for each $a$, $b$ and $c\in A$

$\{a,\{b,c\}\}=\{\{a,b\},c\}+(-1)^{(\vert a\vert+1)(\vert b\vert+1)}
\{b,\{a,c\}\}.$

\noindent b)  the product and the Leibniz bracket satisfy the following relation
called the Poisson relation:
$$\{a,bc\}=\{a,b\}c+(-1)^{(\vert a\vert+1)\vert b\vert}b\{a,c\}.$$
\end{defin}
\begin{proposition}\label{Gerstenhaber algebra associe a une dga}
Let $A$ be a graded algebra equipped with an operator
$d:A_n\rightarrow A_{n+1}$ such that $d\circ d=0$ and such
that $d$ is a derivation.
Then $A$ equipped with the {\it derived bracket} defined
by~\cite[(2.8)]{Kosmann-Schwarzbach:PoissontoGerst} 
$$[a,b]_d:=(-1)^{\vert a\vert +1}[da,b]$$
is a generalized Loday-Gerstenhaber algebra.
\end{proposition}
\begin{proof}
Since $A$ is an associative graded algebra, the bracket $[-,-]$ defined
by $$[a,b]:=ab-(-1)^{\vert a\vert\vert b\vert}ba,$$ is a Lie bracket.
Since $d$ is a derivation for the associative product of $A$,
$d$ is a derivation for the Lie bracket $[-,-]$.
Therefore by~\cite[Proposition 2.1]{Kosmann-Schwarzbach:PoissontoGerst},
the derived bracket $[-,-]_d$ satisfies the graded Jacobi identity
and $d$ is a derivation for the derived bracket $[-,-]_d$. Since  $[-,-]_d$ does not satisfy in general anticommutativity,
$[-,-]_d$ is only a Leibniz bracket in the sense of Loday~\cite{Loday:algebreLeibniz}, and not a Lie bracket in general.
The Lie bracket $[-,-]$ satisfies the Poisson relation:
$$[a,bc]=[a,b]c+(-1)^{(\vert a\vert+1)\vert b\vert}b[a,c].$$
Therefore since $[a,-]_d$ is the derivation $(-1)^{\vert a\vert +1}[da,-]$, the Leibniz
bracket $[-,-]_d$ also satisfies the Poisson
relation~\cite[Proposition 2.2]{Kosmann-Schwarzbach:PoissontoGerst}:
$$[a,bc]_d=[a,b]_dc+(-1)^{(\vert a\vert+1)\vert b\vert}b[a,c]_d.$$
\end{proof}
\begin{rem}\label{crochet de Lie associe a une adg}
In Proposition~\ref{Gerstenhaber algebra associe a une dga}, if instead, we define the bracket by
$$[a,b]_d:=ad(b)-(-1)^{(\vert a\vert +1)(\vert b\vert +1)}bd(a)
$$
then $[-,-]_d$ satisfies anti-commutativity and Jacobi: $[-,-]_d$ is a Lie bracket
\footnote{We could not find this Lie bracket in the litterature. So this Lie algebra structure might
be new.} of degre $+1$.
But this time, $[-,-]_d$ does not satisfy the Poisson relation.
Note that again $d$ is a derivation for $[-,-]_d$.
\end{rem}
\begin{proof}
Let $a\in A_{x-1}$, $b\in B_{y-1}$ and $c\in C_{z-1}$ be three elements of $A$ of
degres $x-1$, $y-1$ and $z-1$. Then
\begin{multline*}
[a,[b,c]_d]_d=ad(bdc)-(-1)^{zy}ad(cdb)\\-(-1)^{xy+xz}b(dc)(da)+(-1)^{xy+xz+yz}c(db)(da),
\end{multline*}
\begin{multline*}
[[a,b]_d,c]_d=a(db)(dc)-(-1)^{xy}b(da)(dc)\\+(-1)^{zx+zy}cd(adb)+(-1)^{zx+zy+xy}cd(bda)
\end{multline*}
and
\begin{multline*}
(-1)^{xy}[b,[a,c]_d]_d=(-1)^{xy}bd(adc)-(-1)^{xy+xz}bd(cda)\\-(-1)^{yz}a(dc)(db)+(-1)^{yz+xz}c(da)(db).
\end{multline*}
Since $d$ is a derivation and $d^2=0$, $d(adb)=(da)(db)$. Therefore we have the Jacobi identity:
$$
[a,[b,c]_d]_d=[[a,b]_d,c]_d+(-1)^{xy}[b,[a,c]_d]_d.
$$
Since $[da,b]_d=(da)(db)$ and $[a,db]_d=-(-1)^{x(y+1)}(db)(da)$,
$$
d([a,b]_d)=(da)(db)-(-1)^{xy}(db)(da)=[da,b]_d+(-1)^x[a,db]_d.
$$
This means that $d$ is a derivation for $[-,-]_d$.
\end{proof}
\begin{ex}(interior derivation)\label{interior derivation}
Let $A$ be an associative graded algebra.
Let $\tau\in A_1$ such that $\tau^2=0$.
Then $d:=[\tau,-]$ is a derivation of the associative product and $d\circ d=0$.
Therefore, we can apply the previous proposition.
In this case, we denote the derived bracket $[a,b]_d$ simply by $[a,b]_\tau$
and~\cite[Example p. 1250]{Kosmann-Schwarzbach:PoissontoGerst}
$$
[a,b]_\tau=(-1)^{\vert a\vert+1}[[\tau,a],b]=[[a,\tau],b].
$$
\end{ex}
\begin{cor}~\cite[Beginning of Section 2.4]{Kosmann-Schwarzbach:PoissontoGerst}\label{endomorphism Gerstenhaber}
Let $E$ be a graded ${\Bbbk}$-module equipped with an operator
$B:E_n\rightarrow E_{n+1}$ such that $B\circ B=0$.
Then $\text{End}(E)$ equipped with the derived bracket
$[a,b]_B=[[a,B],b]$ is a generalized Loday-Gerstenhaber algebra.
\end{cor}
\begin{proof}
Apply Proposition~\ref{Gerstenhaber algebra associe a une dga}
and Example~\ref{interior derivation}, to  $\text{End}(E)$
equipped with the composition product.
\end{proof}
\begin{theor}(implicit in~\cite[p. 1269-70 pointed by Krasilshchik]{Kosmann-Schwarzbach:PoissontoGerst})\label{injection d'une BV-algebre dans End}
Let $A$ be a Batalin-Vilkovisky algebra.
The morphism of graded algebras induced by left multiplication
$$
\Psi:A\rightarrow \text{End}(A), a\mapsto l_a
$$
is an injective morphism of generalized Loday-Gerstenhaber algebras.
\end{theor}
\begin{proof}
Since $A$ is a graded module equipped with an operator $\Delta:A_n\rightarrow A_{n+1}$
such that $\Delta\circ \Delta=0$, by Corollary~\ref{endomorphism Gerstenhaber}
applied to $A$ and to $B=-\Delta$, $\text{End}(A)$ equipped with the
{\it derived bracket} $[f,g]_{-\Delta}=[[f,-\Delta],g]$ is a generalized
Loday-Gerstenhaber algebra. By Proposition~\ref{equivalence BV derived bracket},
$$
l_{\{a,b\}}=-[[l_a,\Delta],l_b]=[[l_a,B],l_b]
$$
Therefore $\Psi$ is a morphism of generalized Loday-Gerstenhaber algebra
\end{proof}
\begin{theor}\label{hochschild morphism de gerstenhaber}
Let $A$ be a differential graded algebra.
Then
 
1) $\text{End} HH_*(A,A)$ equipped with the derived bracket
$$[a,b]_B=[a,B],b]$$ is a generalized Loday-Gerstenhaber algebra.

2) The morphism of graded algebras induced by the action
$$\Phi:HH^*(A,A)\rightarrow \text{End}HH_*(A,A),\quad a\mapsto i_a,$$
is a morphism of generalized Loday-Gerstenhaber algebra.
In particular, its image $\Phi(HH^*(A,A))$ is a Gerstenhaber algebra.
\end{theor}
\begin{proof}
Since Connes boundary $B:HH_*(A,A)\rightarrow HH_{*+1}(A,A)$ satisfies $B\circ B=0$,
by Corollary~\ref{endomorphism Gerstenhaber}, we obtain 1).

Since $i_{ab}=i_a\circ i_b$ (equation~(\ref{action a gauche}))
and $i_{\{a,b\}}=[[i_a,B],i_b]=[i_a,i_b]_B$,
$\Psi$ is a morphism of generalized Gerstenhaber-Loday algebra.

Since $£HH^*(A,A)$ is a Gerstenhaber algebra, $\Phi(HH^*(A,A))$ is also a Gerstenhaber algebra.
\end{proof}
\begin{rem}
If $A$ is a differential graded algebra satisfying the hypotheses of Proposition~\ref{BV algebre homologie de Hochschild}, the morphism $\Phi:HH^*(A,A)\hookrightarrow \text{End}HH_*(A,A)$ of
Theorem~\ref{hochschild morphism de gerstenhaber} is injective
and can be identified with the morphism $\Psi$ of Theorem~\ref{injection d'une BV-algebre dans End}
for the Batalin-Vilkovisky algebra $HH^*(A,A)$.
\end{rem}
\bibliography{Bibliographie}
\bibliographystyle{amsplain}
\end{document}